\documentclass[a4paper,12pt]{article}
\usepackage{amssymb,amsmath,amsfonts,amsthm}

\numberwithin{equation}{section}
\usepackage[left=1 in, right=1 in,top=1 in, bottom=1 in]{geometry}

\providecommand{\abs}[1]{\left\vert#1\right\vert}
\providecommand{\norm}[1]{\left\Vert#1\right\Vert}
\providecommand{\pnorm}[2]{\left\Vert#1\right\Vert_{L^{#2}}}
\providecommand{\pnormspace}[3]{\left\Vert#1\right\Vert_{L^{#2}(#3)}}

\providecommand{\Rn}[1]{\mathbb{R}^{#1}}

\providecommand{\csubset}{\subset\subset}

\def\wstar{\overset{*}{\rightharpoonup}}
\def\nab{\nabla}
\def\dt{\partial_t}

\def\hal{\frac{1}{2}}
\def\lep{\lambda_\varepsilon}
\def\ep{\varepsilon}

\def\ale{\abs{\log \ep}}

\def\a{\alpha}
\def\({\left(}
\def\){\right)}
\def\l|{\left|}
\def\r|{\right|}
\def\ep{\varepsilon}
\def\mr{\mathbb{R}}

\def\mc{\mathbb{C}}

\def\p{\partial}
\def\xib{\xi_0}
\def\hb{h_0}
\def\fb{\psi_0}
\def\xb{X_0}
\def\phib{\phi_0}
\def\lep{\ale}

\def\nab{\nabla}

\def\om{\Omega}
\def\io{\int_{\Omega}}
\def\bo{\partial \Omega}
\def\hal{\frac{1}{2}}

\def\rest{\hskip 1pt{\hbox to 10.8pt{\hfill\vrule height 7pt width 0.4pt depth 0pt\hbox{\vrule height 0.4pt
width 7.6pt depth 0pt}\hfill}}}

\def\evalu{\hskip 1pt{\hbox to 2pt{\hfill \vrule height -6pt width 0.4pt depth0pt}}}

\DeclareMathOperator{\curl}{curl}
\DeclareMathOperator{\diverge}{div}
\DeclareMathOperator{\supp}{supp}
\DeclareMathOperator{\dist}{dist}

\newtheorem{lem}{Lemma}[section]
\newtheorem{cor}[lem]{Corollary}
\newtheorem{prop}[lem]{Proposition}
\newtheorem{thm}[lem]{Theorem}
\newtheorem{remark}[lem]{Remark}

\title{Ginzburg-Landau vortex dynamics with pinning and strong applied currents}
\author{Sylvia Serfaty\footnote{Supported by an EURYI Award}\, and Ian Tice\footnote{Supported by an NSF
Postdoctoral Research Fellowship}}

\begin{document}

\maketitle

\begin{abstract}
We study a  mixed heat and Schr\"odinger  Ginzburg-Landau evolution equation  on a bounded two-dimensional domain  with   an electric current applied on the boundary and a pinning potential term. This is meant to model a superconductor subjected to an applied electric current and electromagnetic field and containing impurities. Such a current is expected to  set the vortices in motion, while the pinning term drives them toward minima of the pinning potential and ``pins'' them there.  We derive the limiting dynamics of a finite number of vortices  in the limit of a large  Ginzburg-Landau parameter, or  $\ep \to 0$, when the intensity  of the electric current and applied magnetic field on the boundary scale like $\lep$.    We show that the limiting velocity of the vortices is the sum of a   Lorentz force, due to the current, and a pinning force. We state an analogous result for a model Ginzburg-Landau equation without magnetic field but with forcing terms.   Our proof provides a unified approach to various proofs of dynamics of Ginzburg-Landau vortices.
\end{abstract}

\noindent
{\bf keywords: } Ginzburg-Landau, vortices, vortex dynamics,  pinning, critical current

\section{Introduction}

\subsection{The model}
In this paper we study the  dynamics of vortices in a superconductor with applied magnetic field  and electric current in addition to  possible pinning effects, under the following 
 mixed heat plus  Schr\"{o}dinger (or complex) flow in a bounded  two-dimensional domain:
\begin{equation}\label{tdgl}
 \begin{cases}
  (\alpha + i \beta \ale ) (\dt u + i\Phi u) = \Delta_{A} u + \frac{u}{\ep^2} (b-\abs{u}^2) & \text{in }\Omega \\
  \sigma (\dt A+ \nab \Phi) = \nab^\bot h + (iu,\nab_{A} u) & \text{in }\Omega \\
  h = H_{ex} & \text{on } \partial \Omega \\
  \nab_{A} u \cdot \nu = i u J_{ex} \cdot \nu & \text{on }\partial \Omega.
 \end{cases}
\end{equation}  
Here $\om$ is the bounded two-dimensional domain representing the region occupied by the superconducting sample. The unknown functions are the triple $(u, A,\Phi)$, where $u: \om \to \mc$ is the ``order parameter,'' $A: \om \to \mr^2$ is a vector potential of the magnetic field, itself given by $h:=\curl A$, and $\Phi: \om \to \mr$ is the scalar potential associated to the electric field, itself given by $E:= -(\dt A+ \nabla \Phi)$.  This is a gauge theory, i.e. $(u, A, \Phi)$ are only known up to gauge-transformations of the form $u \mapsto u e^{i \xi}$, $A \mapsto A + \nabla \xi$, $\Phi \mapsto \Phi - \dt \xi$ where $\xi$ is smooth. The covariant gradient $\nabla_A$ denotes $\nabla - i A$.  For a vector $X \in \Rn{2}$ we write $X^\bot = (-X_2,X_1)$, and for the perpendicular gradient we write $\nab^\bot h = (\nab h)^\bot$.  Also,  $( \cdot, \cdot )$ denotes the scalar product in $\mathbb{C}$ defined by $(a, b)= \mathcal{R}(a) \mathcal{R}(b)+ \mathcal{I}(a) \mathcal{I}(b)$, and $(a, X)$ for $a\in \mathbb{C}$ and  $X\in \mathbb{C}^2$ stands for the  vector in $\mr^2 $ with components $(a, X_1) $ and $(a, X_2)$.

The function $b(x)$ is interpreted as a pinning potential. We assume that $b:\bar{\Omega} \rightarrow \Rn{}$ is a smooth function satisfying
\begin{equation}\label{b_lower_bound}
  0 < \inf_{x \in \Omega} b(x) \le b(x)  \le \sup_{x\in \Omega} b(x) < \infty.
\end{equation} 
The situation without pinning corresponds to the case $b\equiv 1$.

Let us explain the meaning  of the various parameters in the equation.
We assume $\alpha>0$, $\sigma >0$, and $\beta\in \mr$. When $\beta=0$ and $b \equiv 1$, these equations are the Gorkov-Eliashberg system (see  \cite{ge}), which are the standard gauge-invariant heat flow version of the Ginzburg-Landau equation. The case $\a=0$, $\beta>0$ would correspond to a pure gauge-invariant Schr\"odinger flow. Here, for the sake of generality, we consider  $\a>0$ and $\beta$ real, which corresponds to a mixed flow or ``complex Ginzburg-Landau,''  also commonly considered in the modeling of superconductivity \cite{dorsey,kopnin}.

The parameter $\ep>0$ is  a small parameter, equal to the inverse of $\kappa$, the ``Ginzburg-Landau parameter'' in superconductivity, which is a material constant defined as the ratio between two characteristic length scales. We will be interested in the asymptotic limit $\ep \to 0$, corresponding to ``extreme type-II superconductors.''  The parameter $\sigma$ is called the conductivity.  Note that the parameters $\a, \beta, \sigma$ as well as the function $b(x)$ are assumed to be independent of $\ep$.

The boundary conditions are what make this equation quite specific: they are meant to account for an applied normal current, as well as an applied magnetic field.
 To account for an incoming flow  of normal current, the applied (or exterior) field $H_{ex}$ has to be inhomogeneous  on the boundary. Then the  incoming current $I_{ex}$ is given by the static Maxwell equation \begin{equation}\label{static_maxwell}
 \nab^\bot H_{ex} = -I_{ex}.
\end{equation}
These vector fields are defined a priori in the whole $\Omega^c$ (complement of $\Omega$), but  only the data of $H_{ex}$ on the boundary is needed in the equation. The data only provide the information of $I_{ex}\cdot \nu$, which is relevant as the normal component of current. The vector field $J_{ex}$ has the same nature  as an incoming current; it can serve to model an applied voltage or surface charges. For a  discussion of these choices of boundary conditions, we refer to Tice \cite{tice_2}, where they were introduced and justified.  Note that the more common situation of the equation with applied magnetic field of intensity $h_{ex}$ on the boundary but no ``applied current'' can be retrieved by setting $H_{ex}= h_{ex}$ (spatially constant) on $\bo $ (then $I_{ex}=0$) and $J_{ex}=0$. To simplify the dependence on $\ep$ we make the structural assumptions that 
\begin{equation}
  H_{ex} =  \ale H, \text{ and } J_{ex} = \ale J
\end{equation}
for $H, J $ smooth  from $\partial \Omega $ to $\Rn{}$, $\Rn{2}$ respectively. 
In \cite{tice_2}  the case where $J_{ex}=J$ and $H_{ex}=H$ are independent of $\ep$ was treated. Some  larger applied fields and currents were also treated, but the arguments could not be extended to fields as strong as the $\ale$ scale that we consider here.  As a result of the presence of these boundary conditions,  the dynamics are no longer dissipative (or even conservative), i.e. the energy of the system can increase in time.  We have also added the presence of the pinning weight $b$ and the mixed flow, which were not considered in \cite{tice_2}. This adds more generality, but is also quite relevant for the modeling. To explain this let us mention more details of the physics.

Superconductors are particular alloys that lose their resistivity when below a critical temperature, allowing for permanent supercurrents  that circulate without loss of energy.  However, in the presence of  applied fields or currents,  point vortices may appear:  these  can be seen as the  zeros of the order parameter, which all carry  an integer topological charge called degree.  As $\ep \to 0$ the vortices become point-like.   

When a current is applied, it flows through the superconductor,  inducing a Lorentz force that makes the vortices  move, which in turn disrupts the flow of the permanent supercurrents.  This is an important problem in practical applications, where a steady flow of supercurrent is essential.  To counter this effect, a common technique is to introduce impurities in the material, which create ``pinning sites'' that pin down the vortices and prevent them from moving, at least when the pinning is strong enough relative to current. The impurities are modeled by the nonconstant function $b(x)$, with the effect being that vortices  are attracted to the local minima of $b$.  One then wishes to understand at which point the current-induced Lorentz force is strong enough to unpin the vortices and set them in motion; such a current is known as the ``critical current'' in the physics literature.  For a deeper discussion of pinning and critical currents, and of the physics in general, we refer to \cite{tinkham,campbell,chap_her,vinokur} and the references therein.

As we shall explain in Section \ref{glsimple}, the analysis we develop here can also be applied to treat the simpler model equation, which we call ``Ginzburg-Landau with forcing,'' given by 
 \begin{equation}
\label{eqsimpl}
(\a+ i \lep \beta)  \dt u_\ep = \Delta u_\ep + \frac{u_\ep}{\ep^2}(1-|u_\ep|^2) +  \nab h  \cdot \nab u_\ep +  2 i\lep Z \cdot 
\nab u_\ep+f_\ep u_\ep
\end{equation}
with, say, homogeneous Neumann or Dirichlet boundary condition.  Here we will assume that $Z$, $h$, and $f_\ep$ are given smooth functions from $\Omega$ into $\Rn{2}$, $\Rn{}$, and $\Rn{}$, respectively (note that despite the slightly confusing notation, 
 $Z$ and $h$ have nothing to do with  the  currents or  magnetic fields of \eqref{tdgl}).  In fact, the first step for studying \eqref{tdgl} is to make a change of unknown  functions that ``removes'' the boundary condition while transforming the equation into one similar to \eqref{eqsimpl}.

We note that the well-posedness of the Cauchy problem for \eqref{tdgl} as well as \eqref{eqsimpl} can easily be shown by adapting the arguments of \cite{tice_2}.

\subsection{Previous work}

There have been numerous works on the dynamics of a finite number of vortices in various flows for Ginzburg-Landau. Each time the goal is to derive the limiting law, as $\ep \to 0$, for motion of the $n$ vortex points, i.e. a system of $n$ coupled ODEs, at least before the first time of collision of the vortices under that law.

The first results  of this type were  those of Lin \cite{lin_1} and Jerrard-Soner \cite{js} for the heat flow of Ginzburg-Landau 
without magnetic field (i.e. no gauge $A, \Phi$):
\begin{equation}\label{heat}
\dt u=\Delta u + \frac{u}{\ep^2}(1-\abs{u}^2),\end{equation} and with fixed Dirichlet boundary condition,
then Jerrard-Colliander \cite{jc} treated the corresponding Schr\"odinger dynamics. This required
the well-preparedness assumption $E_\ep(u)\le \pi n \lep +O(1)$ (resp. $\pi n \lep +o(1)$ for Schr\"odinger)  where $E_\ep$ is the Ginzburg-Landau energy without magnetic field and  $n$ is the initial number of vortices,  all of which have degree $\pm 1$. After accelerating time by the factor $\lep$, the limiting dynamical law of vortices  as $\ep \to 0$  under the heat flow \eqref{heat} is 
\begin{equation}
\dot{a}_i= - \nab_i W(a_1, \dotsc, a_n) 
\end{equation}
(respectively $\dot{a}_i^\bot= - \nab_i W(a_1, \dotsc, a_n)$ in the Schr\"odinger case and without accelerating time)  where $W$ is a function of interaction between the vortices called the ``renormalized energy,'' introduced in \cite{bbh}.  The case of the heat flow for the gauged equations with spatially homogeneous applied magnetic field was treated by Spirn  in \cite{spirn} for small fields and  Sandier-Serfaty \cite{ss_gamma} for larger fields. This corresponds to setting $b\equiv 1$, $\beta=0$, $H_{ex}=h_{ex}$ and $J_{ex}=0$ in \eqref{tdgl} above and scaling $h_{ex}$ with $\ep$.
Complex  flows started to attract attention recently: the dynamics were derived for the Ginzburg-Landau equation without magnetic field by Miot \cite{miot} in the case of the whole plane, and Kurzke-Melcher-Moser-Spirn \cite{kurzkespirn} in a bounded domain,  and for the  gauged equation   by Kurzke-Spirn in \cite{kurzkespirngauge}.

In the case of pinning, the only complete, rigorous results are due to Lin \cite{lin_2}, who derived the vortex dynamics without gauge and with a different model of pinning than we consider.  Indeed, \cite{lin_2} derives the vortex  dynamics from the equation 
\begin{equation}
\dt u = \frac{1}{b} \diverge(b \nab u) + \frac{u(1-\abs{u}^2)}{\ep^2}. 
\end{equation}
For our specific pinning model there are no complete rigorous results, only partial results by Jian-Song \cite{jiangsong} in the case without gauge.  They correctly guess the limiting dynamical law, and they show that if any vortices persist, then they must concentrate their energy near the vortex paths, which provides evidence of pinning.  However, they do not prove that the original vortices actually persist in time or that no new vortices nucleate, and without this information we do not see how they can fully derive the vortex motion law.  These papers were preceded by formal results by Chapman-Richardson  \cite{cr} on the motion law of a   vortex line in three dimensions for the full magnetic model.

There are only a few results available in the mathematics literature for the problem with applied current, and none consider the effect of pinning.  The case  $I_{ex} \neq 0$ and $J_{ex}=0$ was studied in  \cite{chap_her, du_2, du_3, du_gray}, where  numerical and formal asymptotic results established evidence of current-induced Lorentz forcing in the vortex dynamics.  A stability analysis of the normal state ($u=0$) in a model with  applied current and mixed Dirichlet-Neumann boundary conditions was performed in \cite{almog}.   For a 1-D model of a superconducting wire with applied current, the existence of time-periodic solutions was studied in \cite{rub_stern_zum}.  The rigorous study of the vortex dynamics with boundary current (either $I_{ex}$ and $J_{ex}$)  was completed for the first time in \cite{tice_2}.

Our study here also provides a relatively simple and unified approach to several of the situations mentioned above, combining  several  ingredients (pinning, applied  current and field, mixed flow). We will present the method of proof in Section \ref{glsimple}.

All the studies mentioned above make some ``well-preparedness'' assumption, i.e. assume that the initial energy is not larger than $\pi n\lep$ (in some weaker or stronger form) where $n$ is the initial number of vortices.   The resulting dynamical laws remain valid only for as long as the vortices do not collide or exit the domain.  A much  subtler analysis is required to lift these assumptions and extend the dynamics past collision times. This has been done only in the case of the heat flow without gauge in the series of papers \cite{bos1,bos2} for dynamics in $\Rn{2}$ and \cite{serf_1,serf_2} for a bounded domain.  In the present paper we will not attempt such an analysis.

\subsection{Main result and interpretation}

Before stating the main result we need to introduce various  auxiliary functions with respect to which we make our change of unknown functions.
Let $\phib$ be the solution to 
\begin{equation}\label{phi_def}
 \begin{cases}
  -\sigma \Delta \phib + \alpha b \phib = 0 & \text{in }\Omega \\
   \nab \phib \cdot \nu = \sigma^{-1}(b J - I )\cdot \nu & \text{on } \partial \Omega.
 \end{cases}
\end{equation} 
Note that $\phib$ exists, is unique, and is smooth.  Also, $-I \cdot \nu = - \nab H \cdot \tau$, for $\tau=\nu^\bot$ the unit tangent on $\partial \Omega$. We will work in a fixed gauge in which $\Phi_\ep = \ale \phib$. 
 
Define $\hb:\Omega \rightarrow \Rn{}$ to be the solution to the PDE 
\begin{equation}\label{h_def}
\begin{cases}
  -\diverge\left(\frac{\nab \hb}{b}\right) +  \hb = -\sigma \nab^\bot \frac{1}{b}\cdot \nab \phib & \text{in }\Omega \\
   \hb  = H & \text{on } \partial \Omega.
 \end{cases}
\end{equation}
Again, $\hb$ exists, is unique, and is smooth.  Define $\xib:\Omega \rightarrow \Rn{}$ to be the solution to the PDE
\begin{equation}\label{k_def}
\begin{cases}
  \Delta \xib = \hb & \text{in }\Omega \\
   \xib  = 0 & \text{on } \partial \Omega.
 \end{cases}
\end{equation}
Then we define $\xb:\Omega \rightarrow \Rn{2}$ via
\begin{equation}\label{X_def}
 \xb = \nab^\bot \xib,
\end{equation}
which implies that 
\begin{equation}\label{propeX}
\begin{cases}
\curl{\xb}=\hb,  \quad \diverge{\xb}=0   & \text{in }  \Omega\\
 \xb \cdot \nu =0,  \quad\curl{\xb} = H & \text{on } \partial \Omega.
\end{cases}
\end{equation}
Finally, define $\fb:\Omega \rightarrow \Rn{}$ to be the solution to the PDE 
\begin{equation}\label{f_def}
\begin{cases}
  \Delta \fb = \diverge\left( \frac{\sigma \nab \phib - \nab^\bot \hb}{b} \right) & \text{in }\Omega \\
   \nab \fb \cdot \nu  = J\cdot \nu & \text{on } \partial \Omega.
 \end{cases}
\end{equation}
The PDE is well-posed because it satisfies the necessary compatibility condition (see \eqref{f_wp} for a more precise statement).  The functions $\fb$ and $\xb$ are chosen in this way so that $\sigma \nab \phib - \nab^\bot \curl{\xb} = b(\nab \fb - \xb)$ (see Lemma \ref{X_f_relation} for proof).  Note that all these functions depend only on the parameters of the problem ($b$, $\a, \beta, \sigma, J, H$) and on the domain $\om$, but not on $\ep$.  The $0$ index is used  to emphasize this fact.

We transform the equations by making the change of unknown functions \begin{equation}\label{changefunctions}
v_\ep= \frac{u_\ep}{\sqrt{b}} e^{-i \lep \fb} \qquad B_\ep= A_\ep - \lep \xb.
\end{equation}
The energy of $(v, B)$ is defined by 
\begin{equation}\label{F}
F_\ep(v,B) = \hal \io b |\nab_B v|^2 + \frac{b^2}{2\ep^2} (1-\abs{v}^2)^2 + |\curl B|^2. \end{equation}
However, the energy conditions are better  phrased on the modified energy, which we will see only differs from $F_\ep(v, B)$ by a term that is $o(1)$:
\begin{multline}
\tilde{F}_\ep(v,B) := \int_\Omega 
\hal \( b \abs{\nab_B v}^2 + \frac{b^2}{2\ep^2} (1-\abs{v}^2)^2 + \abs{\curl B}^2 \) \\
+ 
 \int_\Omega \frac{(\abs{v}^2-1)}{2}(b \lep^2 \abs{\nab \fb- \xb}^2 - \sqrt{b} \Delta \sqrt{b} - \beta \ale^2
 b \phib )+ \int_{\partial \Omega} \frac{(\abs{v}^2-1)}{4} \nab b \cdot \nu.
\end{multline}

 Finally, we need to introduce the vorticity measure, or Jacobian, of the configuration. Following for example \cite{ss_prod}, let us introduce the ``space-time Jacobian'' of a configuration $(v, B)$ by setting  
\begin{equation}
\mathcal{J} = d \( (iv, \partial_t  v  )dt +(iv, d_B v) + B \)
\end{equation}
in the language of  differential forms, where $d_B $ denotes $d_{space}+ iB$ with $B$ regarded as a one-form, $d_{space}$ is the exterior derivative with respect to the spatial variables only, and $d$ is the full space-time exterior derivative.
Writing  $\mathcal{J} = V_1 \, dx_1 \wedge dt + V_2  d x_2 \wedge dt + \mu dx_1 \wedge dx_2$, we identify the spatial Jacobian (or vorticity) 
\begin{equation}\label{mu_def}
\mu(v,B)= \curl \((iv, \nab_B v)+B\) 
\end{equation}
and the velocity  vector field $V=(V_1, V_2)$
\begin{equation}\label{V_def}
V( v, B)= \nabla (iv, \partial_t v) + \partial_t (iv, \nab_B v) - \partial_t B.
\end{equation}
Since $d\circ d=0$ it holds that $d\mathcal{J}=0$, which implies the continuity equation 
\begin{equation}
\partial_t \mu + \curl  V=0. 
\end{equation}
The vorticity $\mu$ and the velocity vector field $V$ are  typically  measures concentrated at the vortices.  For example, if $\mu(t)= 2\pi \delta_{\gamma(t)}$,  then $V(t)= 2\pi \dot{\gamma}^\perp(t) \delta_{\gamma(t)}$, which shows that in reality $V$ encodes the perpendicular to the actual velocity of the vortex.

We say that  the solution is well-prepared if
\begin{equation}\label{well_prepared_def}
\left\{\begin{array}{rl}
 \mu(v_\ep,B_\ep)(0) & \rightarrow 2\pi \sum_{i=1}^n d_i \delta_{a_i^0} \text{ for }d_i = \pm 1,  \\ [2mm]
\frac{1}{\ale} \tilde{F}_\ep (v_\ep,B_\ep)(0) &= \sum_{i=1}^n \pi b(a_i^0) + o(1). 
\end{array}\right.
\end{equation} where the $a_i^0$ are distinct points.
It is known since \cite{bbh} that each vortex carries an energy of at least $\pi \lep$.  With pinning this easily translates into the following estimate:  if $\mu(u_\ep, A_\ep) \to 2\pi \sum_{i=1}^n d_i \delta_{a_i}$ 
with $d_i =\pm 1$ and $a_i\in \Omega$ distinct points, then 
\begin{equation}\label{gcv}
\tilde{F}_\ep(u_\ep, A_\ep) \ge \pi \sum_{i=1}^n b(a_i) \lep +o(\lep).
\end{equation} 
We will refer to this as the ``standard $\Gamma$-convergence result,'' but we will not prove it. 
It can be easily shown by adapting results in the literature  (see for example \cite{ass}).  Because of the lower bound \eqref{gcv} it is essential that we assume \eqref{b_lower_bound} so that the  minimal energetic cost of a vortex is $\pi (\inf b) \ale > C \ale$, which is not $o(\ale)$.   Without this assumption, our techniques would be unable to control the number of vortices.

Notice that the  well-preparedness conditions \eqref{well_prepared_def} amount to requiring equality in \eqref{gcv}.  This form of well-preparedness is  relatively weak in comparison to those commonly found in the literature, which usually require bounds by $\pi n \lep +O(1)$.  Again, it is easy to adapt results in the literature to show that, given any set of $n$ distinct points $a_i \in \Omega$,  initial data satisfying \eqref{well_prepared_def} can be constructed.

Our main result on \eqref{tdgl} is
\begin{thm}\label{dynamics-intro}
Let $(u_\ep, A_\ep, \Phi_\ep)$ be solutions to \eqref{tdgl} and choose the gauge $\Phi_\ep = \ale \phib$. Assume the initial data for the solutions are  well-prepared in the sense that \eqref{well_prepared_def} holds, where $(v,B)$ is given by \eqref{changefunctions}.
Then there exists a  time $T_*>0$  and $n$ continuously differentiable functions $a_i : [0,T_*)\to \om$ such that 
the initial vortices move along the trajectories $a_i$ (i.e. $\mu(v_\ep, B_\ep)(t) \to \mu(t)= 2\pi \sum_i d_i \delta_{a_i(t)}$)
and 
\begin{equation}\label{dynalaw}
\begin{cases}
\alpha \dot{a}_i + d_i \beta \dot{a}_i^\bot =  -2 d_i\( \nab^\bot \fb(a_i) - \xb^\bot (a_i)  \) - \nab \log b(a_i)\\
a_i(0)=a_i^0.\end{cases}
\end{equation} 
Moreover, $T_*$ is the smaller of the first  collision time and domain exit time of vortices under this law.  In addition, for all $t \in [0, T_*)$ we have 
\begin{equation}
\lim_{\ep \to 0} \frac{\tilde{F_\ep}(v_\ep, B_\ep) (t)}{\lep} = \pi \sum_{i=1}^n b(a_i(t)). 
\end{equation}
\end{thm}
\begin{remark}
$\text{}$

\begin{enumerate}
\item  The last relation shows that even though the energy of the system can increase, no significant excess energy is created.

\item According to Lemma \ref{X_f_relation} and equation \eqref{propeX}, the dynamical law \eqref{dynalaw} may be rewritten as
\begin{equation}\label{dynamics_phi}
\alpha \dot{a}_i + d_i \beta \dot{a}_i^\bot =  -2 d_i \left(\frac{ \sigma \nab^\bot \phib(a_i) + \nab \hb(a_i)  }{b(a_i)} \right)- \nab \log b(a_i). 
\end{equation}
This form highlights the fact that the first term on the right may be separated into electric and magnetic parts since $\phib$ is the electric potential (our gauge is $\Phi_\ep = \ale \phib$) and $\hb$ is related to the external magnetic field via the boundary condition in \eqref{h_def}.

\item Our analysis may be modified in a straightforward way to handle the case of pinning with some smaller parameter regimes: the case of complex multiplier of the form $(\alpha + i \beta_\ep)$ with $\beta_\ep = O(\ale)$, and the case of fields $J_{ex} = j_{ex} J$ and $H_{ex} = h_{ex} H$ with  $j_{ex} = O(\ale)$ and $h_{ex} = O(\ale)$.  When one of $\beta_\ep$, $h_{ex}$, $j_{ex}$ is smaller, i.e. of order $o(\ale)$, the resulting dynamics are the same as those derived in Theorem \ref{dynamics-intro} by setting $\beta =0$, $H =0$, $J=0$, respectively.  In particular, when $h_{ex}$ and $j_{ex}$ are both $o(\ale)$ but $\beta_\ep = \beta \ale$, we get
\begin{equation}\label{pinning_only}
\alpha \dot{a}_i + d_i \beta \dot{a}_i^\bot =   - \nab \log b(a_i).
\end{equation} 
\end{enumerate}

\end{remark}

In the dynamical law we identify $-\nab \log b(a_i)$  as the pinning force, which is identical to that conjectured  in \cite{cr,jiangsong} and essentially the same as that found with a different pinning model in \cite{lin_2}.  The other term,  $-2 d_i\left( \nab^\bot \fb(a_i) - \xb^\bot (a_i)  \right) := -2 d_i Z^\bot(a_i)$, may be identified as the current-induced Lorentz force.  Setting $b\equiv 1$, $\beta =0$, and employing the definitions \eqref{phi_def} and \eqref{h_def} in the form of the dynamical law \eqref{dynamics_phi} shows that $Z$ is the same as the current-induced force identified in \cite{tice_2}.  If in addition $H_{ex}= h_{ex}$ and $J=0$, then $Z$ is the same as in  \cite{ss_gamma}.  It is a simple matter to rewrite \eqref{dynalaw} as
\begin{equation}\label{dot_solved}
 \dot{a}_i = \frac{\alpha}{\alpha^2 + \beta^2} \left( -2d_i Z^\bot(a_i) - \nab \log b(a_i)\right) - \frac{\beta}{\alpha^2 + \beta^2} \left( 2 Z(a_i) - d_i \nab^\bot \log b(a_i)\right),
\end{equation}
which reveals an interesting feature of the dynamics: one part of the forcing depends on the degree (pushing $d = \pm 1$ vortices in opposite directions), and one part does not (pushing all vortices in the same direction).

From \eqref{pinning_only} it is clear that at least one of $j_{ex}$ and $h_{ex}$ must be of the order $O(\ale)$ in order for anything other than pinning to drive the dynamics.  In this way we identify the order of the critical current as $O(\ale)$.  A more precise identification of the critical current is somewhat more delicate since it relies on the exact form of the applied currents, $J$ and $I$.  One possible way to make this identification is to further assume that $J = \lambda J_0$ and $I = \lambda I_0$ for some fixed $I_0, J_0$ with $\lambda \in (0,\infty)$.  The dynamical law \eqref{dynalaw} is then rewritten as 
\begin{equation}\label{dynamics_lambda}
\alpha \dot{a}_i + d_i \beta \dot{a}_i^\bot =  -2 d_i \lambda \( \nab^\bot \fb(a_i) - \xb^\bot (a_i)  \) - \nab \log b(a_i),
\end{equation}
where $\fb$ and $\xb$ are determined by replacing $J,I$ with $J_0,I_0$ in the PDEs \eqref{phi_def}--\eqref{f_def}. We can then define the critical current to be $\lambda_0 \ale$, where 
\begin{equation}
 \lambda_0 := \inf \{ \lambda>0 \;\vert\;  \text{no solution to }\eqref{dynamics_lambda} \text{ remains confined near the local minima of $b$ } \}.
\end{equation}
Clearly, $\lambda_0$ depends on $J_0,I_0$, and $b$, and is difficult to compute explicitly because of \eqref{phi_def}--\eqref{f_def} unless $J_0,$ $I_0,$ $b$, and $\Omega$ have some simple forms.

In the regime we have chosen, where the applied field and currents have strength $\lep$,  all the forcing terms in the dynamical law \eqref{dynalaw} have equivalent strength, and the interaction between the vortices (the renormalized energy found in all the works above) is negligible compared to them.  Hence, in contrast to the dynamics derived for weak (i.e. $O(1)$) fields in \cite{tice_2}, the inter-vortex interaction disappears for strong fields.  Also, these forces are strong enough to make the vortices move at finite speed, without the need to accelerate time. In a way this makes the analysis simpler as we shall discuss below.

\subsection{Method of the proof and case of the Ginzburg-Landau equation without magnetic field}\label{glsimple}

As in \cite{tice_2}, the first step is to transform the equation via the  change of unknown functions \eqref{changefunctions}.
The function $\fb$ and the vector-field $\xb$ are in fact chosen in such a way that the applied field and current disappear from the boundary condition, but appear as additional terms in the right-hand of the PDEs. Note also that here the pinning has the effect of penalizing $b- \abs{u_\ep}^2$, so with this change we expect $\abs{v_\ep}$ to be close to $1$.
For the sake of simplicity, and to extract the relevant ideas, we now present a sketch of our proof in the case  $A_\ep \equiv 0, \Phi_\ep\equiv 0$, i.e. for the equation without gauge, which can also have an interest in itself.  The change of unknown function  essentially transforms the equation \eqref{tdgl}  for $u_\ep$  into an  equation for $v_\ep$  of the form \eqref{eqsimpl} (for a more accurate equation, see Lemma \ref{specific_reformulation}) with homogeneous Neumann boundary condition, for some vector field $Z (x)$, some function $h(x)$ (in fact $h= \log b$), independent of $\ep$, and some function $f_\ep(x)$ which may depend on $\ep$ but blows up much slower than, say, $1/\ep$ (this is not the optimal condition).  The question is then to understand the effect  of the forcing terms  $\nab h \cdot \nab u_\ep$  and $2i \lep Z \cdot  \nab u_\ep$ in this simple  ``forced'' Ginzburg-Landau 
equation and to show the $f_\ep u_\ep$ term has no influence.

In this simpler setting the vorticity $\mu(u)$ and velocity $V(u)$ can be recomputed from \eqref{mu_def}--\eqref{V_def} as 
\begin{equation}\label{musimple}
\mu(u)= \curl (iu, \nab u)=  2(i \partial_1  u, \partial_2 u)
\end{equation}
and 
\begin{equation}\label{Vsimple}
V(u)=\nabla (iu, \p_t u) - \p_t (iu , \nab u) = 2(\p_t u , i\nab u).
\end{equation}
For simplicity we will denote $\mu_\ep$ for $\mu(u_\ep)$ and $V_\ep$ for $V(u_\ep)$.

The classical starting point to obtain the limiting dynamical law is to  look at the evolution of the local energy density 
\begin{equation}
e_\ep(u)= \hal |\nab u|^2 + \frac{(1-\abs{u})^2}{4\ep^2},
\end{equation}
use the stress-energy tensor, and try to take the limit.  However, here it is much more convenient to use  a weighted energy density 
\begin{equation}
 \tilde{e}_\ep(u_\ep) = e^h \left(e_\ep(u) + \frac{1-\abs{u}^2}{2} f_\ep\right).
\end{equation}
Since $e^h$ appears frequently, we  set  $b =e^h$ so that $h= \log b$,  which is consistent with our notation for pinning.

An easy   computation gives that for any function $u_\ep$,
\begin{equation}\label{evole}
\dt e_\ep(u_\ep) = \diverge (\dt u_\ep, \nab u_\ep) -\( \dt u_\ep, \Delta u_\ep+ \frac{u_\ep}{\ep^2}(1-\abs{u_\ep}^2)  \).
\end{equation} 
If $u_\ep$ is a solution to \eqref{eqsimpl}, we then have 
\begin{multline}\label{evole2}
\dt  e_\ep(u_\ep) =  \diverge (\dt u_\ep, \nab u_\ep)
- \a \abs{\dt u_\ep}^2 + (  \dt u_\ep, \nab u_\ep  \cdot \nab h ) \\
+  \lep( \dt u_\ep,   2 Z  \cdot  i\nab u_\ep  )+ f_\ep(u_\ep, \dt u_\ep) .
\end{multline}
The last term in the right-hand side can be recognized as 
\begin{equation}
\frac{d}{dt} \left( \frac{\abs{u_\ep}^2 -1}{2} f_\ep \right)
\end{equation}
and the one before last as $\lep V_\ep \cdot Z$.  We then deduce that 
\begin{equation}\label{evolet}
\dt \tilde{e}_\ep(u_\ep)=  \diverge \( b  ( \dt u_\ep, \nab u_\ep) \)   - \a b   |\dt u_\ep|^2 
+ \lep b V_\ep \cdot Z.
\end{equation}

We next define the weighted stress-energy tensor associated to $u_\ep$ by 
\begin{equation}\label{Tweight}
T_\ep = b \( \nab u_\ep \otimes  \nab u_\ep - \(e_\ep(u_\ep)  + \frac{1-\abs{u_\ep}^2}{2}  f_\ep\)  I_{2\times 2} \) ,
\end{equation} and a computation yields 
\begin{equation}\label{divt}
\diverge T_\ep = b \big( \Delta u_\ep +  \frac{u_\ep}{\ep^2}(1-\abs{u_\ep}^2) + \nab h \cdot \nab u_\ep +   f_\ep u_\ep , \nab u_\ep \big)
 - \tilde{e}_\ep(u_\ep)  \nab h + b \nab f_\ep \frac{ \abs{u_\ep}^2 - 1} {2} .
\end{equation}
Note that $T_\ep$ is a symmetric $2\times 2$ tensor, so $\diverge T_\ep$ denotes the vector whose coordinates are the divergence of the rows of $T_\ep$.
If  $u_\ep$ is a solution to \eqref{eqsimpl} we deduce 
\begin{multline}\label{divvvtt}
\diverge T_\ep =  b \a (\dt u_\ep, \nab u_\ep ) + b \beta \lep   (i \dt u_\ep, \nab u_\ep  )  - \lep b( 2iZ \cdot \nab u_\ep, \nab u_\ep) 
 - \tilde{e_\ep}(u_\ep) \nab h\\
 + b \nab f_\ep \frac{ \abs{u_\ep}^2 - 1} {2} \\
  = b \a (\dt u_\ep, \nab u_\ep ) - \hal \beta \lep b V_\ep   - b Z^\perp \mu_\ep  - \tilde{e}_\ep(u_\ep) \nab h +  b \nab f_\ep \frac{ \abs{u_\ep}^2 - 1} {2} 
\end{multline} where we used \eqref{musimple}. 
The limiting law is then derived from the two relations \eqref{evolet} and \eqref{divvvtt}, which can be combined to get  
\begin{multline}\label{coupleq}
\dt \tilde{e}_\ep(u_\ep)=  \frac{1}{\a}\diverge\diverge T_\ep   + \frac{\beta}{2 \a} \lep\diverge( b V_\ep)  + \frac{\lep }{\a} \diverge (b Z^\perp \mu_\ep) + \frac{1}{\a} \diverge( \tilde{e}_\ep(u_\ep) \nab h) \\
- \diverge \left(\frac{b}{\a} \nab f_\ep \frac{\abs{u_\ep}^2 -1}{2}\right)
    - \a b   |\dt u_\ep|^2 
+ \lep b V_\ep \cdot Z.
\end{multline}

To be able to take limits in these relations, a priori bounds on all terms, in particular energy bounds are needed. In typical situations (heat or mixed flows) the total energy $\int e_\ep(u_\ep)$ or $\int \tilde{e}_\ep(u_\ep)$ decreases in time, but this is not the case here because of the presence of the forcing terms  which can  bring in energy (this is observable in \eqref{evolet}). However  the ``product estimate'' of \cite{ss_prod} provides control of $V$: it tells us that 
\begin{equation}\label{Vprodest}
\hal \abs{\int_{t_1}^{t_2} \int_\Omega V \cdot X} 
\le \liminf_{\ep \rightarrow 0} \frac{1}{\ale}\left( \int_{t_1}^{t_2}  \int_\Omega \abs{\nab u_\ep \cdot X}^2  \int_{t_1}^{t_2}  \int_\Omega \abs{\dt u_\ep}^2  \right)^{1/2}
\end{equation}
for any smooth vector field $X$, where $V$ is the limit of $V_\ep$ as $\ep \to 0$.
This estimate, inserted in \eqref{evolet}, gives control of the growth of the energy  as in \cite{tice_2}. While in the timescale of \cite{tice_2} the estimate shows the energy increases  by at most a constant in (small) finite time, here we only get that it increases by at most $\eta\lep $ in the (small)  time interval $[0,T_*]$.  This is due to the strength of the forcing terms (now in $\lep$ instead of constant) and pinning terms, and it  is a significant difference: while the $\eta$ can be taken small enough ($<\pi \inf b$) so that  no additional vortex can appear, there can still be an energy of the same order as  the vortex energy floating around. The usual proofs of dynamics mentioned above \cite{lin_1,lin_2,js,spirn,tice_2,miot,kurzkespirn} use the fact that from a priori bounds,   the energy density  $\frac{1}{\lep} e_\ep(u_\ep) $ can only concentrate at the vortex locations and  converges in the sense of measures to $\pi \sum \delta_{a_i(t)}$ where the $a_i(t)$ are the vortex centers.  Here we cannot use this and have to accept the possibility of another term in the limiting energy measure, which is not necessarily concentrated at points,  and with which we work until we eventually can prove it is zero by a Gronwall argument. 
 So we denote $\nu(t)$ the limit in the sense of measures of the energy density 
 $\frac{1}{\lep}\tilde{e}_\ep(u_\ep(t))$; all we know (from the $\Gamma$ convergence  result \eqref{gcv})  is that if there are $n$ limiting  vortices located at $a_i(t)$ then
 \begin{equation}\label{minonu}
 \nu(t)\ge \pi \sum_{i=1}^n  b (a_i(t)).\end{equation}
 
Returning now to the discussion of  taking limits in \eqref{evolet} and \eqref{divvvtt}, we note that in situations where there is bounded order forcing or  no forcing, time must  be accelerated for vortex motion, and one takes the limit of 
$\dt \frac{\tilde{e}_\ep (u_\ep)}{\lep}$ on the one hand, and equates it with that of $\diverge T_\ep$ on the other hand.  The stress-energy tensor $T_\ep$ itself is not bounded as $\ep \to 0$, so it does not have a true limit. However,  $\diverge T_\ep$ has a limit in ``finite parts'' (see \cite{ss_book} Chap. 13 for this viewpoint) and one can show it is equal to the gradient of the ``renormalized energy'' of \cite{bbh}. This is the main force driving the dynamics (possibly supplemented with the forcing of $Z$)  in the dynamics derived in  \cite{tice_2}. It is in particular  this computation in finite parts that makes the proofs \cite{lin_1,lin_2,js,spirn,tice_2,miot,kurzkespirn}   delicate. 

Our present situation is somewhat easier in that sense. Because of the strong forcing, no rescaling in time is necessary, and in \eqref{evolet} and \eqref{divvvtt} it suffices to pass to the limit in $\frac{1}{\lep} \tilde{e}_\ep(u_\ep)$ and $\frac{1}{\lep} \diverge T_\ep$ simultaneously.   Now $\frac{T_\ep}{\lep}$ is easily seen to be bounded since $T_\ep$ is bounded by the energy density, so it immediately  has a weak limit, $T$. The terms in $\abs{u_\ep}^2-1$ go to zero in the limit by the energy control and the assumption $\|f_\ep\|_{C^1}\le \frac{1}{\ep}$.
The other terms in \eqref{evolet} and \eqref{divvvtt} all have limits when normalized by $\lep$, thanks to the energy bound and the compactness results on $\mu_\ep$ and $V_\ep$ provided, for example, by \cite{ss_prod}. In particular $\mu_\ep \to \mu= 2\pi \sum_{i=1}^n d_i \delta_{a_i(t)}$ while $V_\ep \to V= 2\pi \sum_{i=1}^n d_i \dot{a}_i ^\bot\delta_{a_i(t)}$ where the $a_i(t)$ are the (continuous) vortex trajectories, and where the number of vortices and their degrees remain constant on $[0,T_*]$.  We may then assume, up to extraction, that 
\begin{equation} 
\frac{T_\ep}{\lep} \to T, \, \frac{\diverge{T_\ep}}{\ale} \to \diverge{T},  \frac{b(\dt u_\ep, \nab u_\ep)}{\ale} \to p, \text{ and } \frac{\a b |\dt u_\ep|^2 }{\lep }\to \zeta 
 \end{equation}
 in the weak sense of measures in $\Omega \times [0,T_*]$.
Dividing by $\lep$ in \eqref{evolet} and \eqref{divvvtt} and taking the limit as $\ep \to 0$, we thus find
\begin{equation}\label{eqlimite1}
 \dt \nu  = \diverge{p} -  \zeta + b V \cdot Z
\end{equation}
and
\begin{equation}\label{eqlimite2}
 \diverge{T} = \alpha p - \hal \beta b V - b \mu Z^\bot - \nu \nab h.
\end{equation}

The next ingredient is to perform a Lebesgue decomposition of the various measures with respect to the vorticity measure $\mu  = 2\pi \sum_{i=1}^n d_i \delta_{a_i(t)} dt$.  This yields 
\begin{equation}\label{eqlimite3}
\begin{split}
\nu & = \nu_0 (t) dt + \sum_{i=1}^n \nu_i(t) \delta_{a_i(t)} dt, \; T = T_0 + \sum_{i=1}^n  T_i(t) \delta_{a_i(t)} dt, \\
 \diverge{T} &= S_0 + \sum_{i=1}^n S_i(t) \delta_{a_i(t)}dt ,\; p = p_0 + \sum_{i=1}^n p_i(t) \delta_{a_i(t)} dt, \text{ and } \\
\zeta  &= \zeta_0  + \sum_{i=1}^n \zeta_i(t) \delta_{a_i(t)} dt, 
\end{split}
\end{equation}
where $\nu_i(t), T_i(t), S_i(t), p_i(t), \zeta_i(t)$ are functions of time and $\nu_0(t)dt, T_0, S_0, p_0, \zeta_0$ are singular with respect to $\mu$.  These quantities are interrelated by \eqref{eqlimite1} and \eqref{eqlimite2}, and indeed from the decomposition \eqref{eqlimite2} we have that
\begin{equation}\label{eqlimite4}
 S_0 = \alpha p_0 - \nu_0 \nab h \text{ and } S_i = \alpha p_i  - \beta \pi b(a_i)  d_i \dot{a}_i ^\bot  - 2\pi b(a_i) d_i Z^\bot(a_i)  - \nu_i \nab h(a_i).
\end{equation}
The latter equation in principle gives the velocity of the vortices, hence the dynamical law; however, in our situation we only know from \eqref{minonu}  that $\nu_i \ge \pi b(a_i)$, but we do not know the precise values of $\nu_i$,  nor do we know $p_i$ or $S_i$.

To determine $p_i$ we appeal to the equation \eqref{eqlimite1}.  If we formally plug the decompositions \eqref{eqlimite3} into \eqref{eqlimite1} and throw away all but the concentrated parts, we find that
\begin{equation}\label{eqlimite5}
\dt\(\sum_{i=1}^n \nu_i \delta_{a_i}\) 
 =  \diverge\left( \sum_{i=1}^n p_i \delta_{a_i} \right)- \sum_{i=1}^n \zeta_i \delta_{a_i} + 2\pi  \sum_{i=1}^n b(a_i) d_i \dot{a}_i^\bot \cdot Z(a_i) \delta_{a_i}.
\end{equation}
We then observe that 
\begin{equation}
\dt \left(\sum_{i=1}^n \nu_i  \delta_{a_i } \right)= -\diverge\left(\sum_{i=1}^n \nu_i  \dot{a}_i \delta_{a_i }\right)    + \sum_{i=1}^n \dt \nu_i  \delta_{a_i } 
\end{equation}
in the sense of distributions.  Using this in \eqref{eqlimite5} and equating  the terms that  are divergences of Dirac deltas, we find  that
\begin{equation}\label{eqlimite6}
 p_i = -\nu_i \dot{a}_i \text{ for }i=1,\dotsc,n.
\end{equation}
This formal calculation can be made rigorous by integrating \eqref{eqlimite1} against an appropriately chosen test function with support  that moves with the vortices, and indeed we can prove that \eqref{eqlimite6} actually holds.  

We can also determine that $T_i=0$ for each $i=1,\dotsc,n$.  If we  use the decompositions \eqref{eqlimite3} in \eqref{eqlimite2}, then  we see that
\begin{equation}
\diverge{T_0} + \sum_{i=1}^n \diverge\left(   T_i \delta_{a_i} \right)  = S_0 + \sum_{i=1}^n S_i \delta_{a_i}.
\end{equation}
Testing this against appropriately chosen vector fields then allows us to deduce that $T_i=0$ for $i=1,\dotsc,n$.

The next step is to prove that $\nu_i = \pi b(a_i)$, $\nu_0=0$, and $\zeta_0=0$ by a Gronwall argument, the key to which is a pair of estimates for $\zeta_i$ and $S_i$.  The first estimate, 
\begin{equation}\label{corprodes}
\int_{0}^t \sum_{i=1}^n \zeta_i \ge \int_0^t \a \pi^2 \sum_{i=1}^n b^2(a_i) \frac{|\dot{a}_i|^2}{\nu_i},
\end{equation}
 is a simple corollary of the product estimate \eqref{Vprodest}.   The second estimate, 
\begin{equation}\label{eqlimite7}
 \int_0^t \sum_{i=1}^n S_i \cdot \dot{a}_i \le \hal \int_0^t \int_\Omega \zeta_0 + \int_0^t C\left(1+\sum_{i=1}^n \abs{\dot{a}_i}^2 \right) \int_\Omega \nu_0,
\end{equation}
may be derived by relating the pairings $\diverge{T}\cdot \Xi$ and $T:D\Xi$ (here $A:B = \sum_{lk} A_{lk}B_{lk}$) through  the divergence theorem and using an appropriate test vector field $\Xi$.  We integrate the relation \eqref{eqlimite1}  over $\Omega$ and between time $0$ and $t$ to find 
\begin{equation}
\io \nu(t)- \io \nu(0) + \int_0^t \io \zeta =  \int_0^t\sum_{i=1}^n 2 \pi d_i b(a_i) Z(a_i) \cdot \dot{a}_i^\bot.
\end{equation}
Into this equation we insert \eqref{eqlimite4}, \eqref{eqlimite6}, \eqref{corprodes},  the well-preparedness assumption (implying $\nu(0)= \pi \sum_{i=1}^n b(a_i(0))$),  and the relation $\pi \sum_{i=1}^n b(a_i(t))= \pi \sum_{i=1}^n b(a_i(0)) + \int_0^t \pi \sum_{i=1}^n \nabla b \cdot \dot{a}_i$; 
after rearranging a little  we are led to
\begin{multline}
\io \nu_0(t)  +   \sum_{i=1}^n (\nu_i(t) - \pi b(a_i(t)))  + \int_0^t \int_\Omega \zeta_0 \le \alpha \int_0^t \sum_{i=1}^n \abs{\dot{a}_i}^2\left( \nu_i  - \frac{\pi^2 b^2(a_i )}{\nu_i}\right) \\
 +\int_0^t \sum_{i=1}^n  (\dot{a}_i \cdot  \nab h(a_i))(\nu_i  - \pi b(a_i)) + \int_0^t \sum_{i=1}^n S_i \cdot \dot{a}_i.
\end{multline}
Using the fact that $\nu_i \ge \pi b(a_i)$ in conjunction with  \eqref{eqlimite7}, we find that
\begin{multline}
\io \nu_0(t)  +   \sum_{i=1}^n (\nu_i(t) - \pi b(a_i(t)))  + \hal \int_0^t \int_\Omega \zeta_0 \\ \le  \int_0^t   C\left(1+\sum_{i=1}^n \abs{\dot{a}_i}^2 \right)  \left[\sum_{i=1}^n(\nu_i  - \pi b(a_i))  + \int_\Omega \nu_0 \right].
\end{multline}
A simple application of Gronwall's lemma, using the fact that $a_i \in H^1$, then allows us to conclude that $\zeta_0=0$ and $\nu_i(t)=\pi b(a_i(t))$, $\nu_0(t)=0$  for all time, i.e. we find  a posteriori that no significant excess energy develops and that $\zeta$ only concentrates.

In the final step, we use the fact that $\nu_0 =0$ and $\zeta_0=0$ to show that $T_0 =0$ and $S_0=0$.  We can then return to \eqref{eqlimite3} to compare
\begin{equation}
0=  \diverge\left(T_0 + \sum_{i=1}^n T_i(t) \delta_{a_i(t)}dt \right) = \diverge{T} = S_0 + \sum_{i=1}^n S_i(t) \delta_{a_i(t)}dt = \sum_{i=1}^n S_i(t) \delta_{a_i(t)}dt.
\end{equation}
Hence $S_i =0$ a.e. for all $i=1,\dotsc,n$.  Since we now know $\nu_i$, $p_i$, and $S_i$, we may return to \eqref{eqlimite4} to deduce the dynamical law
\begin{equation}
\label{dynlawint}
\a \dot{a}_i + d_i \beta \dot{a}_i^\bot = -  2 d_i Z^\perp(a_i)   - \nab h(a_i).
\end{equation}
We find that in the original equation \eqref{eqsimpl} the forcing term 
$\nab h \cdot \nab u_\ep$ translates into a force $-\nab h(a_i)=-  \nab \log b(a_i)  $ acting on each vortex and pushing it toward the local minima of $b$, while the forcing term $2i \lep Z\cdot \nab u_\ep$ translates into a Lorentz-type force 
$- 2d_i Z^\perp(a_i)$. 
We can  write this result as 
\begin{thm}\label{th2}
Let $u_\ep$ be a solution to \eqref{eqsimpl} where $Z(x)$ is a smooth vector field, $h(x)$ a smooth function and $f_\ep(x)$ a function satisfying $\|f_\ep\|_{C^1}\le \frac{1}{\ep}$ and assume the well-preparedness condition 
\begin{equation} 
\left\{\begin{array}{rl}
 \mu(u_\ep)(0) & \rightarrow 2\pi \sum_{i=1}^n d_i \delta_{a_i(0)} \text{ for }d_i = \pm 1,  \\ [2mm]
\io \tilde{e}_\ep(u_\ep)(0)&= \pi  \sum_{i=1}^n e^h(a_i) \lep +o(\lep),
\end{array}\right.
\end{equation} 
where the $a_i(0)\in \Omega$ are distinct points.
Then  there exist $n$ continuously differentiable functions $a_i:[0,T_*)\rightarrow \Omega$  such that   the vortices move along the trajectories $a_i$ (i.e. $\mu(t)= 2\pi \sum_i d_i \delta_{a_i(t)}$) solving 
\begin{equation}
\a \dot{a}_i + d_i \beta \dot{a}_i^\bot = -  2 d_i Z^\perp(a_i)   - \nab h(a_i). 
\end{equation}
Moreover, $T_*$ is the smallest of  the first collision time and the first exit time of vortices under this law.\end{thm}


Note that a simple adaptation  of this allows for the treatment of the case of the Ginzburg-Landau equation with pinning 
\begin{equation}
(\a + i \beta \lep) \dt u_\ep= \Delta u_\ep + \frac{u_\ep}{\ep^2} (b-\abs{u_\ep}^2), 
\end{equation}
where $\min b>0$,
since the change of unknown function  $v_\ep= \frac{u_\ep}{\sqrt{b}}$ transforms it into \begin{equation}
(\a+ i \beta \lep) \dt v_\ep= \Delta v_\ep+ \frac{bv_\ep}{\ep^2} (1- \abs{v_\ep}^2) + \nab \log b \cdot \nab v_\ep+ 
v_\ep\frac{ \Delta \sqrt{b} }{\sqrt{b}}
\end{equation}
and yields the dynamical law 
\begin{equation}
\a \dot{a}_i + d_i \beta \dot{a}_i^\bot =   - \nab \log b(a_i),
\end{equation}
which is identical to \eqref{pinning_only}.

In more usual settings of order 1 forcing terms and logarithmic time rescaling, the method of proof above can also be applied, modulo the computation of  the limit of $\diverge T_\ep$ in finite parts, and it thus gives a unified approach to derive the dynamical law for the heat flow, mixed heat plus Schr\"odinger flow, with or without pinning, with or without an ``applied current" type forcing term.  It avoids having to choose ``clever" test-functions, as done in particular for mixed flows, and allows for the possibility  of excess-energy developping in time.

We note that in our analysis it is crucial that  the  currents or forcing terms are of strength $\ale$, but again this scaling is particularly relevant since it is precisely that for which the pinning force and the electromagnetic forces are of the same order. The case of stronger currents raises new difficulties and is still an open question.

The paper is organized as follows.  In Section \ref{sec2} we present the change of functions and choice of auxiliary functions which allow to transform the equation \eqref{tdgl} into one resembling \eqref{eqsimpl}.
Section \ref{sec3} contains various estimates controlling $1-\abs{v_\ep}$ and $B_\ep$ that ensure their compactness; it can be skipped in a first reading.  In Section \ref{sec4} we employ the method of \cite{tice_2} to show that, while the energy does not decrease, it cannot increase too quickly.  Sections \ref{sec5} and \ref{sec6} contain the core of the proof of Theorem \ref{dynamics-intro}; it is there that we derive the convergence results and dynamical law, roughly following the method outlined in Section \ref{glsimple}.

\section{Reformulating the equations}\label{sec2}
In this section we give details on the change of functions that serve to transform the equations. The idea follows \cite{tice_2}, but the pinning term complicates the choice of functions.

\subsection{Equations}

Studying the triple $(u_\ep,A_\ep,\Phi_\ep)$ is not convenient because of the pinning and boundary terms.  We reformulate the equations to remove the appearance of the applied fields from the boundary conditions.

\begin{lem}\label{general_reformulation}
Suppose  $\psi:\Omega \rightarrow \Rn{}$ and $X:\Omega \rightarrow \Rn{2}$ are both smooth and satisfy the boundary conditions
\begin{equation}\label{g_r_01}
 \begin{cases}
  \curl{X} = H & \text{on }\partial \Omega \\
  X \cdot \nu = 0 & \text{on }\partial \Omega \\
  \nab \psi \cdot \nu = J \cdot \nu & \text{on }\partial \Omega.
 \end{cases}
\end{equation} 
Let $v_\ep = u_\ep e^{-i \ale \psi}/\sqrt{b}$ and $B_\ep = A_\ep - \ale X$.  Write $Z  =\nab \psi - X$ and $Z_\ep= \ale(\nab \psi- X)$. Then $(v_\ep,B_\ep,\Phi_\ep)$ solve 
\begin{multline}\label{g_r_02}
 (\alpha + i \ale \beta )(\dt v_\ep + i\Phi_\ep v_\ep) = \Delta_{B_\ep} v_\ep + \frac{b v_\ep}{\ep^2} (1-\abs{v_\ep}^2) + \nab \log{b} \cdot \nab_{B_\ep} v_\ep + 2 i Z_\ep \cdot  \nab_{B_\ep} v_\ep \\
+ iv_\ep  \frac{\diverge(b Z_\ep )}{b} + v_\ep\left(\frac{\Delta \sqrt{b}}{\sqrt{b}} -  \abs{Z_\ep}^2   \right),
\end{multline}
\begin{equation}\label{g_r_03}
 \sigma(\dt B_\ep + \nab \Phi_\ep) = \nab^\bot \curl{B_\ep} + b (iv_\ep,\nab_{B_\ep} v_\ep) +   \ale \nab^\bot \curl{X} +\abs{v_\ep}^2  b Z_\ep 
\end{equation}
in $\Omega$, along with the boundary conditions
\begin{equation}\label{g_r_04}
 \begin{cases}
  \curl{B_\ep} = 0  & \text{on }\partial \Omega \\
  \nab_{B_\ep} v_\ep \cdot \nu = -\hal v \nab\log{b} \cdot \nu & \text{on }\partial \Omega.
 \end{cases}
\end{equation}
\end{lem}
\begin{proof}
The PDEs \eqref{g_r_02} and \eqref{g_r_03} follow directly from the equations \eqref{tdgl} and the definitions of $v_\ep, B_\ep$.  The boundary conditions \eqref{g_r_04} follow from \eqref{g_r_01} and the boundary conditions in \eqref{tdgl}.
\end{proof}

\subsection{Choice of subtracted fields}

We now seek to find a choice for $\psi$ and $X$ to use in Lemma \ref{general_reformulation} that leads to some optimal cancellation in the PDEs \eqref{g_r_02}--\eqref{g_r_03}.   To this end, we first  fix a gauge.  Define $\phib:\Omega \rightarrow \Rn{}$ to be the solution to \eqref{phi_def}.
Note that $\phib$ exists, is unique, and is smooth.    We now fix a gauge in which $\Phi_\ep = \ale \phib$.    

\begin{lem}\label{gauge_fix}
 It is possible to change gauges so that $\Phi_\ep = \ale \phib$ and so that $B_\ep(0)$ satisfies the Coulomb gauge, i.e. $\diverge{B_\ep(0)} =0$ and $B_\ep(0) \cdot \nu =0$ on $\partial \Omega$.
\end{lem}
\begin{proof}
 The result is identical to that of Lemma 2.4 in \cite{tice_2}.
\end{proof}

With this choice of gauge, we now choose the $\psi$ and $X$ to work with.   Define $\hb:\Omega \rightarrow \Rn{}$ to be the solution to \eqref{h_def}.
Again, $\hb$ exists, is unique, and is smooth.  Define $\xib:\Omega \rightarrow \Rn{}$ to be the solution to \eqref{k_def}.
Then we define $\xb:\Omega \rightarrow \Rn{2}$ via \eqref{X_def},
which implies \eqref{propeX}.
  Finally, define $\fb:\Omega \rightarrow \Rn{}$ to be the solution to \eqref{f_def}.
This PDE is well-posed since \eqref{static_maxwell}, \eqref{phi_def}, and \eqref{h_def} imply that
\begin{equation}\label{f_wp}
 \int_\Omega \diverge\left( \frac{\sigma \nab \phib - \nab^\bot \hb}{b}\right) = \int_{\partial \Omega} \frac{\sigma \nab \phib - \nab^\bot \hb}{b} \cdot \nu = 
\int_{\partial \Omega} \frac{(bJ-I)\cdot \nu + I \cdot \nu}{b} =
 \int_{\partial \Omega} J \cdot \nu.
\end{equation}
The reason for defining $\fb$ and $\xb$ in this manner is seen in the following lemma.

\begin{lem}\label{X_f_relation}
 Let  $\phib$, $\hb$, $\xb$, and $\fb$ be as defined in \eqref{phi_def}--\eqref{f_def}.  Then
\begin{equation}
\sigma  \nab \phib - \nab^\bot \curl{\xb} = b(\nab \fb -\xb).
\end{equation}
\end{lem}
\begin{proof}

 Since $\diverge{\xb}=0$, the definition of $\fb$ \eqref{f_def} implies that
\begin{equation}
 \diverge(\nab \fb - \xb) = \diverge\left( \frac{\sigma \nab \phib - \nab^\bot \hb}{b} \right),
\end{equation}
so that by Poincar\'{e}'s lemma, 
\begin{equation}\label{x_r_1}
 \nab \fb - \xb = \frac{\sigma \nab \phib - \nab^\bot \hb}{b} + \nab^\bot \chi
\end{equation}
for some $\chi: \Omega \rightarrow \Rn{}$.  Taking the $\curl{}$ of  equation \eqref{x_r_1} yields
\begin{equation}
 \Delta \chi = -\hb -\sigma  \nab^\bot \frac{1}{b} \cdot \nab \phib + \diverge\left(\frac{\nab \hb}{b} \right) = 0
\end{equation}
since $\hb$ satisfies \eqref{h_def}.  Taking the dot product of \eqref{x_r_1} with the boundary normal yields 
\begin{equation}
 -\partial_\tau \chi = \nab^\bot \chi \cdot \nu = J\cdot \nu - \frac{(bJ-I) + I}{b} \cdot \nu =0,
\end{equation}
where $\tau = \nu^\bot$ is the boundary tangent.   This implies that $\chi$ is a constant on the boundary, and then inside, and so the result follows since $\curl{\xb} = \hb$.
\end{proof}

This property leads to a nice cancellation in the equations of Lemma \ref{general_reformulation}, which allows to obtain an equation \eqref{s_r_00} of the form \eqref{eqsimpl} (in a gauged version).

\begin{lem}\label{specific_reformulation}
Let $\fb$, $\xb$ be as in \eqref{f_def} and \eqref{X_def} and suppose $(u_\ep,A_\ep)$ are solutions to \eqref{tdgl} in the $\Phi_\ep = \ale \phib$ gauge.  Define $Z_\ep = \ale Z$ for
\begin{equation}
 Z := \nab \fb - \xb
\end{equation}
and also define
\begin{equation}
 f_\ep = \frac{\Delta \sqrt{b}}{\sqrt{b}} - \abs{Z_\ep}^2  + \beta \ale^2 \phib.
\end{equation}
Then $v_\ep = u_\ep e^{-i \ale \fb}/ \sqrt{b}$ and $B_\ep = A_\ep - \ale \xb$ solve
\begin{equation}\label{s_r_00}
(\alpha + i \ale \beta) \dt v_\ep  = \Delta_{B_\ep} v_\ep + \frac{b v_\ep}{\ep^2} (1-\abs{v_\ep}^2) + \nab \log{b} \cdot \nab_{B_\ep} v_\ep + 2 i Z_\ep \cdot  \nab_{B_\ep} v_\ep 
 + v_\ep f_\ep,
\end{equation}
\begin{equation}\label{s_r_01}
 \sigma \dt B_\ep = \nab^\bot h_\ep' + b (iv_\ep,\nab_{B_\ep} v_\ep) +  (\abs{v_\ep}^2-1)  b Z_\ep 
\end{equation}
in $\Omega$, along with the boundary conditions
\begin{equation}\label{s_r_02}
 \begin{cases}
  h_\ep' = 0  & \text{on }\partial \Omega \\
  \nab_{B_\ep} v_\ep \cdot \nu = -\hal v_\ep \nab \log{b} \cdot \nu & \text{on }\partial \Omega.
 \end{cases}
\end{equation}
Here we have written $h_\ep' = \curl{B_\ep}$.
\end{lem}
\begin{proof}
We apply Lemma \ref{general_reformulation} with $f = \fb$ and $X= \xb$.  According to Lemma \ref{X_f_relation}, 
\begin{equation}
 \diverge(b(\nab \fb-\xb)) = \sigma \Delta \phib = \alpha b \phib.
\end{equation} Replacing $\Phi_\ep$ by $\ale \phib$, 
this yields the cancellation of all the terms multiplying $i v_\ep$ in the equation \eqref{g_r_02}, which gives \eqref{s_r_00}.  A similar application of Lemma \ref{X_f_relation} gives \eqref{s_r_01}.
\end{proof}

\begin{remark}\label{Z_smooth}
The vector field $Z$ defined in Lemma \ref{specific_reformulation} is smooth and $\norm{Z}_{C^2(\Omega)} < \infty$.  This follows immediately from its definition and the smoothness of the solutions to \eqref{phi_def}--\eqref{f_def}.
\end{remark}

\section{A priori estimates and compactness}\label{sec3}
We introduce  the pinned free-energy density 
\begin{equation}\label{gdef}
 g_\ep(u,A) := \hal \left( b \abs{\nab_A u}^2 + \frac{b^2}{2\ep^2}(1-\abs{u}^2)^2 + \abs{\curl{A}}^2\right),
\end{equation} and the free energy
\begin{equation}\label{Fdef}
F_\ep(u, A) = \io g_\ep(u, A).\end{equation}

We are interested in proving a priori estimates and compactness on the magnetic field $B_\ep$. 
The choices of gauge and $\fb, \xb$  give rise to nice properties for it. Some difficulty goes into controlling the terms in $\abs{v_\ep}$ since the maximum principle does not hold for this mixed flow equation, and as such it is not guaranteed that $\abs{v_\ep}\le 1$.

\begin{lem}\label{diverge_control}
Let $(v_\ep,B_\ep)$ solve \eqref{s_r_00}--\eqref{s_r_02} in the $\Phi_\ep = \ale \phib$ gauge.  Then the vector field $B_\ep$ satisfies
\begin{equation}\label{d_c_00}
 \sigma \dt \diverge{B_\ep} = \alpha b (iv_\ep,\dt v_\ep) + \alpha  \ale (\abs{v_\ep}^2-1)b \phib +\ale  \dt \left( \frac{\beta b(\abs{v_\ep}^2 -1 )}{2 }\right)
\end{equation}
in $\Omega$, and
\begin{equation}\label{d_c_01}
\sigma \dt B_\ep \cdot \nu = \ale (\abs{v_\ep}^2-1) J\cdot \nu 
\end{equation}
on $\partial \Omega$.  Consequently, for $1< p < 2$ and $q = 2p/(2-p)$
\begin{multline}\label{d_c_02}
 \sigma \pnorm{\diverge{B_\ep(t)}}{p} \le 
\alpha \int_0^t  \pnorm{\sqrt{b} \dt v_\ep(s)}{2}\left( \pnorm{\sqrt{b}}{q} + 
\pnorm{\sqrt{b}}{\infty}  \pnorm{1-\abs{v_\ep}^2}{q}   \right) ds    \\
+ \alpha \ale \int_0^t  \pnorm{b \phib}{q} \pnorm{(\abs{v_\ep(s)}^2-1)}{2} ds  \\
+ \ale \frac{\abs{\beta} \pnorm{b}{\infty}}{2} \left( \pnorm{1 - \abs{v_\ep(t)}^2 }{p} +  \pnorm{1 - \abs{v_\ep(0)}^2 }{p}  \right) 
\end{multline}
and
\begin{equation}\label{d_c_03}
\sigma \pnormspace{B_\ep(t)\cdot \nu}{p}{\partial \Omega} \le  \int_0^t \ale   
\pnormspace{J \cdot \nu}{\infty}{\partial \Omega} \pnormspace{(\abs{v_\ep(s)}^2-1) }{p}{\partial \Omega}ds.
\end{equation}

\end{lem}
\begin{proof}

 The  equation in \eqref{s_r_01} reads 
\begin{equation}\label{d_c_1}
\sigma \dt B_\ep = \nab^\bot h_\ep' + b (iv_\ep,\nab_{B_\ep} v_\ep) + (\abs{v_\ep}^2-1)b Z_\ep. 
\end{equation}
Taking the dot product of this equation with the boundary normal $\nu$ and applying the boundary conditions $h_\ep' = 0$ and $\nab_{B_\ep} v_\ep \cdot \nu =-\hal v_\ep \nab \log b \cdot \nu$  on $\partial \Omega$ \eqref{s_r_02} yields
\begin{multline}
 \sigma \dt B_\ep \cdot \nu = \nab^\bot h_\ep' \cdot \nu + b (iv_\ep,\nab_{B_\ep} v_\ep \cdot \nu) + (\abs{v_\ep}^2-1) \ale ( \nab \fb - X_0)\cdot \nu \\
 = (\abs{v_\ep}^2-1)\ale  J \cdot \nu,
\end{multline}
which is \eqref{d_c_01}. Taking the divergence of \eqref{d_c_1} and employing Lemma \ref{X_f_relation}, we find that
\begin{equation}\label{d_c_2}
\sigma  \dt \diverge{B_\ep} +  \alpha b \ale \phi_0 = \diverge( b (iv_\ep,\nab_{B_\ep} v_\ep) + \abs{v_\ep}^2 b Z_\ep   ).
\end{equation}
On the other hand, taking $(iv_\ep,\cdot)$ with  \eqref{s_r_00} yields the equality
\begin{multline}
 \alpha (iv_\ep,\dt v_\ep)  + \beta \ale (v_\ep,\dt v_\ep) = (iv_\ep, \Delta_{B_\ep} v_\ep + \nab \log b \cdot \nab_{B_\ep}v_\ep + 2 i Z_\ep \cdot \nab_{B_\ep} v_\ep) \\ =
\diverge(iv_\ep,\nab_{B_\ep} v_\ep) + (iv_\ep,\nab_{B_\ep} v_\ep) \cdot \nab \log b + \nab\abs{v_\ep}^2 \cdot Z_\ep 
\end{multline}
so that (again using Lemma \ref{X_f_relation})
\begin{equation}\label{d_c_3}
\alpha b (iv_\ep,\dt v_\ep)  + \beta \ale b \dt \frac{(\abs{v_\ep}^2 -1)}{2} +  \alpha b \ale \abs{v_\ep}^2 \phi_0 = \diverge(b (iv_\ep,\nab_{B_\ep} v_\ep) + \abs{v_\ep}^2 b Z_\ep ). 
\end{equation}
Equating \eqref{d_c_2} and \eqref{d_c_3} gives \eqref{d_c_00}.

Recall that by Lemma \ref{gauge_fix}, at time $t=0$ the vector field $B_\ep(0)$ satisfies $\diverge{B_\ep(0)}=0$ and $B_\ep(0)\cdot \nu =0$.  The estimate \eqref{d_c_03} then follows directly from integrating \eqref{d_c_01} in time from $0$ to $t$.  For \eqref{d_c_02} we first integrate \eqref{d_c_00} in time and then apply the H\"{o}lder inequality $\pnorm{\phi \psi}{p} \le \pnorm{\phi}{2} \pnorm{\psi}{q}$ for $q = 2p/(2-p)$ to bound
\begin{equation}
 \pnorm{b (iv_\ep, \dt v_\ep )}{p} \le   \pnorm{b \abs{v_\ep} \abs{\dt v_\ep} }{p} \le \pnorm{\sqrt{b} \dt v_\ep(s)}{2}\left( \pnorm{\sqrt{b}}{q} + 
\pnorm{\sqrt{b}}{\infty}  \pnorm{1-\abs{v_\ep}^2}{q}   \right)
\end{equation}
and 
\begin{equation}
\pnorm{b \phib (\abs{v_\ep(s)}^2-1) }{p} \le  \pnorm{b \phib}{q} \pnorm{(\abs{v_\ep(s)}^2-1)}{2}.
\end{equation}

\end{proof}

We will use this result with the following. Recall $F_\ep$ is defined in \eqref{Fdef}.

\begin{lem}\label{bndry_conv}
For any $(v_\ep,B_\ep)$ (not necessarily solutions) it holds that for $2 < q < \infty$
\begin{equation}\label{b_c_0}
 \pnorm{1-\abs{v_\ep}}{q} \le C \ep^{2/q} \sqrt{F_\ep(v_\ep,B_\ep)}
\text{ and }
 \pnormspace{1-\abs{v_\ep}}{q}{\partial \Omega} \le C \ep^{1/q} \sqrt{F_\ep(v_\ep,B_\ep)}
\end{equation}
for some universal constant $C>0$.

\end{lem}

\begin{proof}
Write $\rho = \abs{v_\ep}$.  By rewriting  $v_\ep = \rho w_\ep$ and remembering that $0 < \inf b \le 1$, we find that 
\begin{equation}
 \hal \int_\Omega  \abs{\nab \rho}^2 + \frac{(1-\rho^2)^2}{2\ep^2} \le 
\hal \int_\Omega  \abs{\nab \rho}^2 + \rho^2 
\abs{\nab_{B_\ep} w_\ep }^2 + \frac{(1-\rho^2)^2}{2\ep^2} \le C F_\ep(v_\ep,B_\ep).
\end{equation}
Notice that since $\rho \ge 0$ we may bound $(1-\rho)^2 \le (1-\rho^2)^2$, so that 
\begin{equation}
 \pnorm{1-\rho}{2} \le C \ep \sqrt{F_\ep(v_\ep,B_\ep)} \text{ and }  \norm{1-\rho}_{H^1} \le C  \sqrt{F_\ep(v_\ep,B_\ep)}.
\end{equation}
Interpolating between these bounds yields
\begin{equation}\label{b_c_1}
 \norm{1-\rho}_{H^s} \le \pnorm{1-\rho}{2}^{1-s} \norm{1-\rho}_{H^1}^s  \le C \ep^{1-s} \sqrt{F_\ep(v_\ep,B_\ep)}
\end{equation}
for any $s \in(0,1)$.  When $s \in (1/2,1)$ trace theory then gives 
\begin{equation}\label{b_c_2}
 \norm{1-\rho}_{H^{s-1/2}(\partial \Omega)} \le C \norm{1-\rho}_{H^{s}(\Omega)} \le C \ep^{1-s} \sqrt{F_\ep(v_\ep,B_\ep)}.
\end{equation}
Then the bounds \eqref{b_c_0} follow from \eqref{b_c_1} and \eqref{b_c_2} by using the embedding $H^r \hookrightarrow L^q$ for $q = 2n / (n-2r)$ first with $r=s$ and $n=2$ and then with $r=s-1/2$ and $n=1$.
\end{proof}

We will need the following inequality in order to derive some compactness results for $B_\ep$.

\begin{prop}\label{boundary_poincare}
For $1 < p < 2$ and  $\delta >0$ sufficiently small there exists a constant $C>0$ so that 
\begin{equation}\label{bnd_p_0}
\norm{A}_{W^{1/p - \delta,p} } \le C \left( \pnorm{\diverge{A}}{p} + \pnorm{\curl{A}}{p} + \pnormspace{A\cdot \nu}{p}{\partial \Omega} \right)
\end{equation}
for all $A\in W^{1,p}(\Omega ; \Rn{2})$.  Moreover, 
\begin{equation}\label{b_p_01}
 \pnorm{A}{q} \le C \left( \pnorm{\diverge{A}}{p} + \pnorm{\curl{A}}{p} + \pnormspace{A\cdot \nu}{p}{\partial \Omega} \right)
\end{equation}
for $q = 2p/(1+\delta p)$.

\end{prop}

\begin{proof}
We employ the Hodge decomposition $A = \nab \phi + \nab^\bot \psi$
for $\phi$  the solution to 
\begin{equation}
\begin{cases}
 \Delta \phi = \diverge{A}  & \text{in }\Omega \\
 \nab \phi \cdot \nu = A \cdot \nu & \text{on }\partial \Omega
\end{cases}
\end{equation}
and $\psi$ the solution to 
\begin{equation}
\begin{cases}
 \Delta \psi = \curl{A}  &\text{in } \Omega\\
 \psi  = 0 & \text{on } \partial \Omega.
\end{cases}
\end{equation}
The usual elliptic theory \cite{lions_magenes} provides  the existence such  solutions satisfying the estimates
\begin{equation}\label{b_p_1}
 \norm{\psi}_{W^{2,p}} \le C \pnorm{\curl{A}}{p}
\end{equation}
and
\begin{equation}\label{b_p_2}
 \norm{\phi }_{W^{1+ 1/p - \delta,p} } \le C \left( \pnorm{\diverge{A}}{p}  + \pnormspace{A\cdot \nu}{p}{\partial \Omega} \right)
\end{equation}
for any $\delta>0$ sufficiently small.  

Then
\begin{equation}
 \norm{A}_{W^{1/p - \delta,p} } \le \norm{\nab^\bot \psi}_{W^{1/p - \delta,p} } + \norm{\nab \phi}_{W^{1/p - \delta,p} } \le \norm{\psi}_{W^{2,p} } + \norm{\phi }_{W^{1+ 1/p - \delta,p} }
\end{equation}
which together with \eqref{b_p_1} and \eqref{b_p_2} implies  \eqref{bnd_p_0}.  The bound \eqref{b_p_01} follows from \eqref{bnd_p_0} and the embedding $W^{1/p -\delta,p} \hookrightarrow L^q$ for 
\begin{equation}
 \frac{1}{q} = \frac{1}{p} - \frac{1/p-\delta}{n} \text{ with }n=2.
\end{equation}

\end{proof}

We now turn to some a priori bounds on $1-\abs{v_\ep}$ and $B_\ep$.

\begin{prop}\label{B_bound}
Let $(v_\ep,B_\ep)$ solve \eqref{s_r_00}--\eqref{s_r_02}.  Suppose that 
\begin{equation}\label{b_b_01}
\sup_{0 \le s \le t} F_\ep(v_\ep,B_\ep)(s) + \int_0^t \int_\Omega \alpha b \abs{\dt v_\ep}^2 + \sigma \abs{\dt B_\ep}^2 \le K \ale.
\end{equation}
Then the following hold.
\begin{enumerate}
 \item For any $2 < r < \infty$ there exists a constant $C$ depending on $K$ such that 
\begin{equation}\label{b_b_02}
  \sup_{0 \le s \le t} \pnorm{1-\abs{v_\ep(s)}}{r} \le C \ep^{2/r} \sqrt{\ale}
\text{ and }
 \sup_{0 \le s \le t} \pnormspace{1-\abs{v_\ep(s)}}{r}{\partial \Omega} \le C \ep^{1/r} \sqrt{\ale}.
\end{equation}

\item For any $1 < p < 2$ and $\delta >0$ sufficiently small there exists a constant $C>0$ depending on $K$ so that
\begin{equation}\label{b_b_05}
  \sup_{0\le s\le t} \norm{B_\ep}_{W^{1/p - \delta,p}}  \le C\sqrt{\ale}  (1 +  t  \ale \ep^{1/q} )
\end{equation}
for $q = 2p/(1+\delta p)$.

\item For any $2 < q < 4$ there exists a constant $C>0$ depending on $K$ so that 
\begin{equation}\label{b_b_03}
  \sup_{0\le s\le t} \pnorm{B_\ep(s) }{q} \le C\sqrt{\ale}  (1 +  t  \ale \ep^{1/q} ).
\end{equation}

\item There exists a constant $C>0$ depending on $K$ so that
\begin{equation}\label{b_b_04}
 \int_\Omega \abs{v_\ep(t)}^2 \abs{B_\ep(t)}^2 \le C\ale (1 +  t^2  \ale^2 \sqrt{\ep} ).
\end{equation}

\end{enumerate}

\end{prop}
\begin{proof}
To begin, we  note that Lemma \ref{bndry_conv} and \eqref{b_b_01} imply that
\begin{equation}\label{b_b_2}
  \pnormspace{1-\abs{v_\ep}}{r}{\Omega} \le C \ep^{2/r} \sqrt{\ale}
\text{ and }
 \pnormspace{1-\abs{v_\ep}}{r}{\partial \Omega} \le C \ep^{1/r} \sqrt{\ale}
\end{equation}
for all $r>2$, which yields \eqref{b_b_02}.  

To prove the second item, we let $1 < p < 2$ and $\delta >0$ be sufficiently small.  According to Proposition \ref{boundary_poincare} we may bound
\begin{equation}\label{b_b_1}
 \norm{B_\ep}_{W^{1/p - \delta,p}} \le C \left( \pnorm{\diverge{B_\ep}}{p} + \pnorm{\curl{B_\ep}}{p} + \pnormspace{B_\ep  \cdot \nu}{p}{\partial \Omega} \right).
\end{equation}
We will estimate each term on the right hand side of this inequality using Lemma \ref{diverge_control} with the bounds \eqref{b_b_01} and \eqref{b_b_2}.   The bounds \eqref{d_c_03} and \eqref{b_b_2} together with H\"{o}lder's inequality provide the estimate
\begin{equation}\label{b_b_3}
\sigma \pnormspace{B_\ep(t)\cdot \nu}{p}{\partial \Omega} \le  C t \ale^{3/2} \ep^{1/q} 
\end{equation} 
for $q = 2p/(1+\delta p)$.  Similarly, \eqref{d_c_02} gives
\begin{multline}
 \sigma \pnorm{\diverge{B_\ep(t)}}{p} \le  C \left( \int_0^t \pnorm{\sqrt{b} \dt v_\ep}{2}^2 \right)^{1/2} +  C t  \ale^{3/2} \ep^{2/q} +  C  \ale^{3/2} \ep^{2/q} \\ \le 
 C \ale^{3/2}  \ep^{2/q} (1 +  t   ).
\end{multline}
Finally, the energy bound \eqref{b_b_01} and H\"{o}lder imply that for $p<2$, 
\begin{equation}\label{b_b_4}
 \pnorm{\curl{B_\ep(t)}}{p} \le C \pnorm{\curl{B_\ep(t)}}{2} \le C \sqrt{\ale}.
\end{equation}
We may then combine estimates \eqref{b_b_3}--\eqref{b_b_4} with \eqref{b_b_1} to deduce \eqref{b_b_05}

For the third item we now fix $q \in (2,4)$ and choose $1< p < 2$ and $\delta>0$  so that $q = 2p/(1+\delta p)$ and $1/p - \delta >0$.   Then \eqref{b_b_03} follows from \eqref{b_b_05} and the embedding $W^{1/p-\delta,p}(\Omega) \hookrightarrow L^q(\Omega)$.

Now for the fourth item we utilize \eqref{b_b_2} with $r = 2q/(q-2) >2$ to bound
\begin{equation}
 \pnorm{ \abs{v_\ep(t)} \abs{B_\ep(t)} }{2} \le \pnorm{ \abs{v_\ep(t)}  }{r} \pnorm{ B_\ep(t) }{q}
\le C\sqrt{\ale}  (1 +  t  \ale \ep^{1/q} ).
\end{equation}
Squaring this inequality and applying Cauchy's inequality on the right side then gives
\begin{equation}
 \pnorm{ \abs{v_\ep(t)} \abs{B_\ep(t)} }{2}^2  
\le C\ale (1 +  t^2  \ale^2 \ep^{2/q} ) \le C\ale (1 +  t^2  \ale^2 \sqrt{\ep} ),
\end{equation}
where in the last inequality we have used the fact that $2 < q < 4$ implies  $2/q > 1/2$.  This is \eqref{b_b_04}.
\end{proof}

We now parlay these bounds into convergence results.  We begin by recalling a Lemma on compactness in space-time, due to Simon \cite{simon}.

\begin{lem}\label{compactness}
 Suppose $\mathbb{X}, \mathbb{Y}, \mathbb{Z}$ are Banach spaces such that $\mathbb{X} \csubset \mathbb{Y} \hookrightarrow \mathbb{Z}$ and 
\begin{equation*}
 \norm{x}_{\mathbb{Y}} \le C \norm{x}^{1-\theta}_{\mathbb{X}} \norm{x}^{\theta}_{\mathbb{Z}}
\end{equation*}
for some $\theta \in (0,1)$.  Let $1 < p_1, p_2 \le \infty$.  Then each set bounded both in $L^{p_1}([0,T];\mathbb{X})$ and in $W^{1,p_2}([0,T];\mathbb{Z})$ is pre-compact in $L^p([0,T];\mathbb{Y})$ for all $p \le p_1 / (1-\theta)$.  
\end{lem}

With this lemma in hand we can deduce a pair of convergence results.

\begin{prop}\label{B_limits}
Let $(v_\ep,B_\ep)$ solve \eqref{s_r_00}--\eqref{s_r_02}.  Suppose that 
\begin{equation}\label{b_l_01}
\sup_{0 \le s \le T_*} F_\ep(v_\ep,B_\ep)(s) + \int_0^{T_*} \int_\Omega \alpha b \abs{\dt v_\ep}^2 + \sigma \abs{\dt B_\ep}^2 \le K \ale
\end{equation}
for some fixed $T_*>0$.  Fix $2 < r < 4$.  Then  up to the extraction of a subsequence
\begin{equation}\label{b_l_02}
 \frac{B_\ep}{\sqrt{\ale}} \rightarrow B_* \text{ in } L^2((0,T_*);L^r(\Omega)).
\end{equation}
Moreover, 
\begin{equation}\label{b_l_03}
 \frac{\curl{B_\ep}}{\sqrt{\ale}} \rightarrow \curl{B_*} \text{ in } L^2(\Omega \times (0,T_*)),
\end{equation}
and
\begin{equation}\label{b_l_04}
 \int_0^{T_*} \int_\Omega \frac{ \curl{B_\ep}\,\dt B_\ep^\bot}{\ale} \rightarrow \int_0^{T_*} \int_\Omega \curl{B_*} \,  \dt B_*^\bot.
\end{equation}

\end{prop}
\begin{proof}
We will derive these convergence results by applying Lemma \ref{compactness}.  As such we must first verify its hypotheses.  We begin with the convergence of $\bar{B}_\ep := B_\ep / \sqrt{\ale}$.  According to \eqref{b_b_05} of Proposition \ref{B_bound}
\begin{equation}
 \sup_{0 \le t \le T_*} \norm{\bar{B}_\ep(t)}_{W^{1/p-\delta,p}} \le C.
\end{equation}
On the other hand, the energy bound \eqref{b_l_01} implies that 
\begin{equation}\label{258}
 \int_0^{T_*} \norm{\dt \bar{B}_\ep(t)}_{L^2}^2 dt \le C.
\end{equation}
Hence the collection $\{ \bar{B}_\ep \}_\ep$ is uniformly bounded in 
\begin{equation}
L^\infty( (0,T_*) ; W^{1/p-\delta,p}(\Omega)) \cap H^1( (0,T_*) ; L^2(\Omega) ). 
\end{equation}

According to the interpolation result of Theorem 4.3.1/1 of \cite{triebel}, there exists a constant $C>0$ so that
\begin{equation}
 \norm{\bar{B}_\ep}_{W^{s,q}} \le C \norm{\bar{B}_\ep}_{W^{1/p-\delta,p}}^{1-\theta} \norm{\bar{B}_\ep}_{L^2}^{\theta}
\end{equation}
for any $\theta \in (0,1)$, where
\begin{equation}
 s = (1-\theta)\left(\frac{1}{p} -\delta\right) \text{ and }
 \frac{1}{q} = \frac{1-\theta}{p} + \frac{\theta}{2}.
\end{equation}
We chain this inequality together with the embedding 
\begin{equation}\label{b_l_10}
 W^{s,q}(\Omega) \hookrightarrow L^r(\Omega) \text{ for } \frac{1}{r} = \frac{1}{q} - \frac{s}{2} \text{ with } s>0
\end{equation}
to get
\begin{equation}
 \norm{\bar{B}_\ep}_{L^r} \le C \norm{\bar{B}_\ep }_{W^{1/p-\delta,p}}^{1-\theta} \norm{\bar{B}_\ep}_{L^2}^{\theta}
\end{equation}
for 
\begin{equation}\label{b_l_1}
 \frac{1}{r} = \frac{1-\theta}{p} + \frac{\theta}{2}  - \frac{1-\theta}{2} \left( \frac{1}{p} - \delta\right).
\end{equation}
By choosing $1< p < 2$, $\delta>0$ small enough, and $\theta \in (0,1)$ we can achieve any $r \in (2,4)$ in \eqref{b_l_1}.  So, for any $r \in(2,4)$ we can apply Lemma \ref{compactness} with $\mathbb{X} = W^{1/p-\delta,p}(\Omega)$, $\mathbb{Y} = L^r(\Omega)$, and $\mathbb{Z} = L^2(\Omega)$ to get \eqref{b_l_02}.

We now turn to the convergence of $\zeta_\ep := \curl{B_\ep}/\sqrt{\ale}$.  Because of \eqref{s_r_01} we may write
\begin{equation}
  \nab^\bot \zeta_\ep = \sigma \dt \bar{B}_\ep - b \frac{(iv_\ep,\nab_{B_\ep} v_\ep) }{\sqrt{\ale}} - (\abs{v_\ep}^2-1)\frac{b Z_\ep}{\sqrt{\ale}}.
\end{equation}
This equation, the energy bound \eqref{b_l_01}, H\"{o}lder's inequality, and \eqref{b_b_02} then allow us to bound, for $1< p<2$
\begin{equation}\label{b_l_2}
 \int_0^{T_*} \pnorm{\nab \zeta_\ep}{p}^2 \le C \int_0^{T_*} \left( \pnorm{\dt \bar{B}_\ep}{2}^2  + \frac{\pnorm{\nab_{B_\ep} v_\ep}{2}^2}{\ale} + 1\right) \le C(1+T_*).
\end{equation}
Similarly, the bound \eqref{b_l_01} and H\"{o}lder give, for $1<p<2$,
\begin{equation}\label{b_l_4}
 \int_0^{T_*} \pnorm{\zeta_\ep}{p}^2 \le C \int_0^{T_*} \frac{F_\ep(v_\ep,B_\ep)}{\ale}  \le C T_*.
\end{equation}
We will estimate $\dt \zeta_\ep$ spatially in $H^{-1}(\Omega) = (H^1_0(\Omega))^*$.  Let $\psi \in L^2((0,T_*); H_0^1(\Omega))$.  Then since $\zeta_\ep = \curl{\bar{B}_\ep}$ we may estimate
\begin{equation}
\int_0^{T_*} \int_\Omega \psi \dt \zeta_\ep = \int_0^{T_*} \int_\Omega -\nab^\bot \psi \cdot  \dt \bar{B}_\ep \le \left(\int_0^{T_*} \pnorm{\dt \bar{B}_\ep}{2}^2 \right)^{1/2} \left(\int_0^{T_*} \norm{\psi}_{H^1}^2 \right)^{1/2},
\end{equation} which implies, with \eqref{258},  upon taking the supremum over all such $\psi$, that 
\begin{equation}\label{b_l_3}
 \int_0^{T_*} \norm{ \dt \zeta_\ep }_{H^{-1}}^2 \le \int_0^{T_*} \pnorm{\dt \bar{B}_\ep}{2}^2 \le C.
\end{equation}   
Hence, from \eqref{b_l_2}, \eqref{b_l_4}, and \eqref{b_l_3} we see that for any  $1 < p < 2$, the collection $\{ \zeta_\ep \}$ is uniformly bounded in 
\begin{equation}
 L^2((0,T_*) ; W^{1,p}(\Omega)) \cap H^1( (0,T_*); H^{-1}(\Omega)).
\end{equation}
To apply Lemma \ref{compactness} we fix $1 < p < 2$ and apply Theorems 4.3.1/1 and 4.8.2 of \cite{triebel} to see that
\begin{equation}
 \norm{\zeta_\ep}_{W^{s,q}}  \le C \norm{\zeta_\ep }_{W^{1,p}}^{1-\theta} \norm{\zeta_\ep}_{H^{-1}}^{\theta}
\end{equation}
for $\theta \in (0,1)$, 
\begin{equation}
 s = 1 - 2 \theta, \text{ and } \frac{1}{q} = \frac{1-\theta}{p} + \frac{\theta}{2}.
\end{equation}
If we choose $\theta = (2p-2)/(3p-2) \in (0,1/2)$, then $s>0$ and we may use the embedding  \eqref{b_l_10} with $r=2$ to bound 
\begin{equation}
  \norm{\zeta_\ep}_{L^2} \le C \norm{\zeta_\ep }_{W^{1,p}}^{1-\theta} \norm{\zeta_\ep}_{H^{-1}}^{\theta}.
\end{equation}
Hence Lemma \ref{compactness} is applicable with $\mathbb{X} = W^{1,p}(\Omega)$, $\mathbb{Y} = L^2(\Omega)$, and $\mathbb{Z} = H^{-1}(\Omega)$, and it provides the convergence  \eqref{b_l_03}.
The convergence of \eqref{b_l_04} is a consequence of \eqref{b_l_03} and the weak-$L^2(\Omega \times (0,T_*))$ compactness of $\dt B_\ep/\sqrt{\ale}$ that follows from the energy bound \eqref{b_l_01}.

\end{proof}

\section{Energy analysis}\label{sec4}
In this section we examine the evolution of the energy and show that it  cannot increase too quickly in time.  We first have the following formula for the  evolution of the energy density \eqref{gdef}.

\begin{lem}\label{en_evolve}
 For any pair $(u,A)$ (not necessarily solutions) it holds that
\begin{multline}
 \dt g_\ep(u,A) = b \diverge(\dt u,\nab_A u) + \curl(\curl{A}(\dt A )) \\
- (\dt u,b \Delta_A u + \frac{b^2 u}{\ep^2}(1-\abs{u}^2)) - (\dt A  )\cdot (\nab^\bot \curl{A} + b (iu,\nab_A u)).
\end{multline}
\end{lem}
\begin{proof}
The result  follows from a direct calculation and the commutator identities
\begin{equation}
 \nab_A  \dt u - \dt \nab_A u = i u \dt A 
\end{equation}
and 
\begin{equation}
 (\partial_2 - iA_2)( \partial_1 - i A_1) u - ( \partial_1 - i A_1) (\partial_2 - iA_2) u = i u \curl{A}.
\end{equation}

\end{proof}

When we apply this to our solutions $(v_\ep,B_\ep)$ in the $\Phi_\ep=\ale \phib$ gauge we are led (as seen in Section \ref{glsimple}) to consider a modification of the free energy density given by
\begin{equation}\label{g_tilde_def}
 \tilde{g}_\ep(u,A) := g_\ep(u,A) + \frac{(1-\abs{u}^2)}{2} b f_\ep,
\end{equation}
where $f_\ep$ is as defined in Lemma \ref{specific_reformulation}. The evolution of the integral of this energy density involves a surface energy term, so we are led to define 
\begin{equation}
 \tilde{F}_\ep(u,A) := \int_\Omega \tilde{g}_\ep(u,A) + \int_{\partial \Omega} \frac{(\abs{u}^2-1)}{4} \nab b \cdot \nu.
\end{equation}

\begin{lem}\label{mod_en_evolve}
Let $(v_\ep,B_\ep)$ be solutions to \eqref{s_r_00}--\eqref{s_r_02} in the $\Phi_\ep=\ale \phib$ gauge.  Then
\begin{multline}\label{m_e_e_01}
 \dt \tilde{g}_\ep(v_\ep,B_\ep) = \diverge\left( b(\dt v_\ep,\nab_{B_\ep} v_\ep) \right)+ \curl(h_\ep' \dt B_\ep ) - \alpha  b \abs{\dt v_\ep}^2 - \sigma \abs{\dt B_\ep}^2 \\
+ b Z_\ep \cdot V(v_\ep,B_\ep).
\end{multline}
In particular, this implies that
\begin{equation}\label{m_e_e_02}
 \dt \tilde{F}_\ep(v_\ep,B_\ep) + \int_\Omega \alpha  b \abs{\dt v_\ep}^2 +\sigma \abs{\dt B_\ep}^2 =  \int_\Omega b  Z_\ep \cdot V(v_\ep,B_\ep).
\end{equation}
\end{lem}
\begin{proof}
We plug  the equations of Lemma \ref{specific_reformulation} into Lemma \ref{en_evolve} to find
\begin{multline}
 \dt g_\ep(v_\ep,B_\ep) = b \diverge(\dt v_\ep,\nab_{B_\ep} v_\ep) + \curl(\curl{B_\ep}(\dt B_\ep )) \\
- (\dt v_\ep, (\alpha + i \ale \beta) b \dt v_\ep  -\nab b \cdot \nab_{B_\ep} v_\ep - 2 i b Z_\ep \cdot  \nab_{B_\ep} v_\ep ) \\
- (\dt B_\ep  )\cdot (\sigma \dt B_\ep - (\abs{v_\ep}^2-1)  b Z_\ep   ) 
+ (\dt v_\ep,v_\ep) b f_\ep .
\end{multline}
The equation \eqref{m_e_e_01} follows from an expansion of these terms and  the equality  (see Lemma 2.12 of \cite{tice_2})
\begin{equation}\label{m_e_e_1}
 (\dt v_\ep,2i\nab_{B_\ep} v_\ep) = V(v_\ep,B_\ep) + (1-\abs{v_\ep}^2)\dt B_\ep.
\end{equation}
The equation \eqref{m_e_e_02} follows from integrating \eqref{m_e_e_01} over $\Omega$ and  applying the divergence theorem and the boundary conditions \eqref{s_r_02}  to get
\begin{multline}
 \int_\Omega \diverge\left( b(\dt v_\ep,\nab_{B_\ep} v_\ep) \right)+ \curl(h_\ep' \dt B_\ep ) =  \int_{\partial \Omega} b (\dt v_\ep, - \hal v_\ep \nab \log{b} \cdot \nu) \\
= -\dt \int_{\partial \Omega} \frac{(\abs{v_\ep}^2 -1)}{4} \nab b \cdot \nu.
\end{multline}\end{proof}

Since our modified energy $\tilde{F}_\ep$ is not positive definite, we must prove a result that controls $F_\ep$ in terms of $\tilde{F}_\ep$ and a negligible error.  This is the content of the next lemma.

\begin{lem}\label{energy_comparison}
 For $\ep$ sufficiently small we have the estimate
\begin{equation}\label{e_c_01}
 F_\ep(v_\ep,B_\ep) \le (1+ C\ep^{1/4} ) \tilde{F}_\ep(v_\ep,B_\ep) + C \ep^{1/4}
\end{equation}
for a constant $C>0$. 
\end{lem}
\begin{proof}
Write $\chi := \nab b \cdot \nu$.  We  apply Cauchy's inequality to find
\begin{multline}\label{e_c_1}
 \abs{\int_\Omega \frac{b (\abs{v_\ep}^2 -1) f_\ep}{2} } \le \hal \int_\Omega \frac{b^2 (\abs{v_\ep}^2 -1)^2 }{4\ep}  + \frac{\ep}{2} \int_\Omega \abs{f_\ep}^2 \\
 \le  \frac{\ep}{2} F_\ep(v_\ep,B_\ep) + \frac{\ep}{2} \int_\Omega \abs{f_\ep}^2 
\le \frac{\ep}{2} F_\ep(v_\ep,B_\ep) + C \ep \ale^4.
\end{multline}
For the boundary term we estimate with H\"{o}lder
\begin{equation}\label{e_c_2}
 \abs{\int_{\partial \Omega} \frac{(\abs{v_\ep}^2-1) \chi}{4}  } \le \pnormspace{\abs{v_\ep}^2 - 1}{4}{\partial \Omega} \pnormspace{\chi }{4/3}{\partial \Omega}.
\end{equation}
By Lemma \ref{bndry_conv} we know that
\begin{equation}
 \pnormspace{\abs{v_\ep} - 1}{q}{\partial \Omega} \le C \ep^{1/q} \sqrt{F_\ep(v_\ep,B_\ep)}
\end{equation}
for any $2 < q < \infty$.  We may then rewrite 
\begin{equation}
 1 - \abs{v_\ep}^2 = -(1-\abs{v_\ep})^2 + 2(1-\abs{v_\ep})
\end{equation}
in order to bound
\begin{multline}\label{e_c_3}
 \pnormspace{\abs{v_\ep}^2 - 1}{4}{\partial \Omega} \le C \left(\pnormspace{\abs{v_\ep} - 1}{8}{\partial \Omega}^2 + \pnormspace{\abs{v_\ep} - 1}{4}{\partial \Omega}    \right) \\
\le C \ep^{1/4}( F_\ep(v_\ep,B_\ep) + \sqrt{F_\ep(v_\ep,B_\ep)}   ).
\end{multline}
We then combine \eqref{e_c_2} and \eqref{e_c_3} and again use Cauchy to get the estimate
\begin{equation}\label{e_c_4}
 \abs{\int_{\partial \Omega} \frac{(\abs{v_\ep}^2-1) \chi}{4}  }  \le C \ep^{1/4}( F_\ep(v_\ep,B_\ep)     +  1  ).
\end{equation}
Now from \eqref{e_c_1} and \eqref{e_c_4} we know that
\begin{multline}
 \tilde{F}_\ep(v_\ep,B_\ep) = F_\ep(v_\ep,B_\ep) + \int_\Omega \frac{b (1- \abs{v_\ep}^2 )}{2} f_\ep + \int_{\partial \Omega} \frac{(\abs{v_\ep}^2-1) \chi}{4} \\
\ge F_\ep(v_\ep,B_\ep) - \frac{\ep}{2} F_\ep(v_\ep,B_\ep) - C \ep \ale^4 - C \ep^{1/4}( F_\ep(v_\ep,B_\ep)     +  1  ) \\
\ge (1- C\ep^{1/4} )  F_\ep(v_\ep,B_\ep) - C \ep^{1/4}
\end{multline}
when $\ep$ is sufficiently small.  This yields the desired inequality.

\end{proof}

As we mentioned, the control of the energy growth comes from the product estimate from \cite{ss_prod}, which we now state in our context.

\begin{prop}\label{prod_est}
Let $(v_\ep,B_\ep)$ solve \eqref{s_r_00}--\eqref{s_r_02} and suppose that
\begin{equation}\label{p_e_0}
\sup_{0 \le t \le T_*} F_\ep(v_\ep,B_\ep)(s) + \int_0^{T_*} \int_\Omega \alpha b \abs{\dt v_\ep}^2 + \sigma \abs{\dt B_\ep}^2 \le C \ale
\end{equation}
for some fixed $T_*>0$.  Then the following hold, up to extraction of a subsequence.
\begin{enumerate}
 \item There exist $\mu \in L^\infty([0,T_*];\mathcal{M}(\Omega)) $ and   $V \in L^2([0,T_*]; (\mathcal{M}(\Omega))^2 )$ such that $\dt \mu + \curl{V} = 0$ and $\mu(v_\ep,B_\ep) \rightarrow \mu$, $V(u_\ep,B_\ep) \rightarrow V$ as $\ep \rightarrow 0$ in $(C^{0,1}(\Omega\times[0,T_*]))^*$.  Here we have written $\mathcal{M}(\Omega) = (C^0(\Omega))^*$ for the space of bounded Radon measures.

 \item For any $[t_1,t_2] \subseteq [0,T_*]$, any $Y \in C^0(\Omega \times [t_1,t_2];\Rn{2})$, $\psi \in C^0(\Omega \times[t_1,t_2])$, we have the bound 
\begin{equation}\label{p_e_1}
\hal \abs{\int_{t_1}^{t_2} \int_\Omega  V  \cdot \psi Y} 
\le \liminf_{\ep \rightarrow 0} \frac{1}{\ale}\left( \int_{t_1}^{t_2}  \int_\Omega \abs{\nab_{B_\ep} v_\ep \cdot Y}^2  \int_{t_1}^{t_2}  \int_\Omega \abs{ \psi \dt v_\ep}^2  \right)^{1/2}.
\end{equation}

\item The mapping $t \mapsto \langle \mu(t),\xi \rangle$ is in $H^1([0,T_*])$ for any $\xi \in C_c^1(\Omega)$, and in particular 
\begin{equation}
 \abs{\langle\mu(t_2),\xi\rangle-\langle\mu(t_1),\xi\rangle} 
\le C \sqrt{t_2-t_1} \liminf_{\ep \rightarrow 0} \frac{1}{\sqrt{\ale}} \left( \int_{t_1}^{t_2} \int_\Omega \abs{\dt  v_\ep}^2  \right)^{1/2}
\end{equation}
for any $[t_1,t_2]\subseteq [0,T]$.

\end{enumerate}
\end{prop}

\begin{proof}
To begin we note that 
\begin{equation}
 \abs{\nab v_\ep} \le \abs{\nab_{B_\ep} v_\ep} + \abs{v_\ep}\abs{B_\ep}
\end{equation}
so that \eqref{b_b_04} of Proposition \ref{B_bound}  and \eqref{p_e_0}, combined with the fact that $b$ is bounded below by a constant independent of $\ep$, imply that 
\begin{equation}
 \sup_{0\le t \le T_*}  \int_\Omega \frac{\abs{\nab v_\ep(t)}^2}{2} + \frac{(1-\abs{v_\ep(t)}^2)^2}{4\ep^2} + \int_0^{T_*} \int_\Omega \abs{\dt v_\ep}^2 \le C \ale.
\end{equation}
This means that Theorem 3 of \cite{ss_prod} is applicable and provides the results of item 1 for $V(v_\ep,0)$ and $\mu(v_\ep,0)$.  However,
the definitions of $\mu, V$ (\eqref{mu_def} and \eqref{V_def})  imply that 
\begin{equation}
\mu(v_\ep, B_\ep)-\mu(v_\ep,0)= \curl \( (1- |v_\ep|^2 )B_\ep\) 
\end{equation}
and 
\begin{equation}
V(v_\ep,B_\ep) - V(v_\ep,0)= \dt \((1-|v_\ep|^2 ) B_\ep\). 
\end{equation}
 The bounds on $F_\ep $ and on $B_\ep$
 \eqref{b_b_03}  imply that $V(v_\ep,B_\ep) - V(v_\ep,0) \rightarrow 0$ and $\mu(v_\ep,0) - \mu(v_\ep,B_\ep) \rightarrow 0$ in the dual of $W^{1, \infty}$ and hence in $(C^{0,1})^*$, so the results in item 1 hold as stated.  

For item 2 we note that Theorem 3 of \cite{ss_prod} also yields the estimate
\begin{equation}\label{p_e_2}
\hal \abs{\int_{t_1}^{t_2} \int_\Omega  V  \cdot \psi Y} 
\le  \left(   \int_{\Omega\times[t_1,t_2] }\nu_Y \right)^{1/2} \left(  \int_{\Omega\times[t_1,t_2]} \nu_\psi  \right)^{1/2}
\end{equation}
where $\nu_Y$ and $\nu_\psi$ are the defect measures of $L^2(\Omega \times [0,T_*])$ convergence of 
\begin{equation}
 \frac{\abs{\nab v_\ep \cdot Y}}{\sqrt{\ale}} \text{ and } \frac{\abs{\psi \dt v_\ep}}{\sqrt{\ale}},
\end{equation}
respectively.  However, 
\begin{equation}
 \abs{\nab_{B_\ep} v_\ep \cdot Y}^2 = \abs{\nab v_\ep \cdot Y}^2 - 2(iv_\ep,\nab v_\ep) \cdot Y (B_\ep \cdot Y) + \abs{v_\ep}^2 \abs{B_\ep \cdot Y}^2, 
\end{equation}
so the convergence results in Proposition \ref{B_limits} and the bounds \eqref{b_b_02} of Proposition \ref{B_bound} guarantee that the defect measure of  $L^2(\Omega \times [0,T_*])$ convergence of 
\begin{equation}
 \frac{\abs{\nab_{B_\ep} v_\ep \cdot Y}}{\sqrt{\ale}} 
\end{equation}
coincides with $\nu_Y$.  Then \eqref{p_e_1} follows from \eqref{p_e_2} since 
\begin{multline}
\left(   \int_{\Omega\times[t_1,t_2] }\nu_Y \right)^{1/2} \left(  \int_{\Omega\times[t_1,t_2]} \nu_\psi  \right)^{1/2} \\
\le   \liminf_{\ep \rightarrow 0} \frac{1}{\ale}\left( \int_{t_1}^{t_2}  \int_\Omega \abs{\nab_{B_\ep} v_\ep \cdot Y}^2  \int_{t_1}^{t_2}  \int_\Omega \abs{ \psi \dt v_\ep}^2  \right)^{1/2}.
\end{multline}
Item 3 follows easily from item 2 as in Theorem 3 of \cite{ss_prod}.
\end{proof}

We then deduce the control on the energy growth, following the same method as in \cite{tice_2}.

\begin{thm}\label{time_bound}
Let $(v_\ep,B_\ep)$ be solutions to \eqref{s_r_00}--\eqref{s_r_02}  in the $\Phi_\ep=\ale \phib$ gauge.   Suppose that the initial data $(v_\ep(0),B_\ep(0))$ are well-prepared in the sense of \eqref{well_prepared_def}. Fix $C_0>0$.  Then there exists a constant $T_0= T_0(C_0)>0$ such that, as $\ep \rightarrow 0$,
\begin{equation}\label{t_b_1}
 \tilde{F}_\ep(v_\ep,B_\ep)(t) < \tilde{F}_\ep(v_\ep,B_\ep)(0) + C_0 \ale  \text{ for all } t\in [0,T_0]
\end{equation}
and 
\begin{equation}
 \int_0^{T_0} \int_\Omega \alpha b\abs{\dt v_\ep}^2 + \sigma \abs{\dt B_\ep}^2  < C_0 \ale.
\end{equation}
\end{thm}

\begin{proof}
The proof is essentially the same as Theorem 2.16 of \cite{tice_2}, so we will present only a sketch of the idea.  For full details see \cite{tice_2}.  For any $t\ge 0$, consider the two conditions
\begin{equation}\label{cond1}
 \tilde{F}_\ep(v_\ep,B_\ep)(s) < \tilde{F}_\ep(v_\ep,B_\ep)(0) + C_0 \ale \text{ for all } s\in [0,t]
\end{equation}
and 
\begin{equation}\label{cond2}
 \int_0^{t} \int_\Omega \alpha  b \abs{\dt v_\ep}^2 + \sigma \abs{\dt B_\ep}^2   < C_0 \ale.
\end{equation}
Define 
\begin{equation}
 \gamma_\ep := \sup \{ t\ge 0 \;\vert\; \text{conditions }\eqref{cond1} \text{ and } \eqref{cond2} \text{ hold}   \}.
\end{equation}
The smoothness of $(v_\ep,B_\ep)$ and the well-preparedness of the initial data guarantee the existence of a time $t_\ep>0$ (depending on $\ep$ and $C_0$) such that both conditions hold for $t_\ep$.  Hence $\gamma_\ep > 0$ for each $\ep$.  We show that actually $\gamma_\ep \ge T_0$  for some $T_0>0$ as $\ep \rightarrow 0$, thereby proving the theorem.  Suppose, by way of contradiction, that 
\begin{equation}
\liminf_{\ep\rightarrow 0}  \gamma_\ep =0.
\end{equation}
We may suppose, up to extraction of a subsequence, that $\gamma_\ep \rightarrow 0$ as $\ep \rightarrow 0$.

Rescale in time at scale $\gamma_\ep$ by defining $w_\ep(x,t) = v_\ep(x,\gamma_\ep t)$ and $C_\ep(x,t) = B_\ep(x,\gamma_\ep t)$.   By the definition of $\gamma_\ep$, the inequalities
\begin{equation}\label{cond1'}
  \tilde{F}_\ep(w_\ep,C_\ep)(t) \le \tilde{F}_\ep(w_\ep,C_\ep)(0) + C_0 \ale \text{ for all } t\in [0,1]
\end{equation}
and 
\begin{equation} \label{cond2'}
 \frac{1}{\gamma_\ep} \int_0^{1} \int_\Omega \alpha b\abs{\dt w_\ep}^2 + \sigma \abs{\dt C_\ep}^2   \le C_0 \ale.
\end{equation}
both hold, but at time $t=1$ one of the inequalities must be an equality since $w_\ep$ and $C_\ep$ are smooth.  However, these inequalities and Lemma \ref{energy_comparison} provide a bound for $F_\ep(w_\ep,C_\ep)$, which in turn allows us to apply Proposition \ref{prod_est} to deduce, since $\gamma_\ep \rightarrow 0$,  that $V(w_\ep,C_\ep)\rightarrow 0$ and $\mu(w_\ep,B_\ep) \rightarrow \mu(0)$.  This information and Lemma \ref{mod_en_evolve}, rescaled in time at scale $\gamma_\ep$, then show that neither inequality \eqref{cond1'} nor \eqref{cond2'}  could be an equality at $t=1$, a contradition.  Indeed from $V(w_\ep,C_\ep)\rightarrow 0$ and \eqref{m_e_e_02} properly scaled in time we deduce 
\begin{equation}
 \tilde{F}_\ep(w_\ep , C_\ep)(1) \le  \tilde{F}_\ep(w_\ep, C_\ep)(0)+o(\lep),
\end{equation}
 so there cannot be equality in \eqref{cond1'}.
On the other hand 
 since the limiting vortices do not move (from $V=0$)  by the $\Gamma$-convergence lower bound \eqref{gcv} we must have 
\begin{equation}
\tilde{F}_\ep(w_\ep , C_\ep)(1)  \ge \pi \sum b(a_i(0))\lep +o(\lep)= \tilde{F}_\ep(w_\ep, C_\ep)(0)+ o(\lep). 
\end{equation}
Combining with \eqref{m_e_e_02} again we deduce 
\begin{equation}
\frac{1}{\gamma_\ep}\int_0^1 \io \a b |\p_t w_\ep |^2+ \sigma |\dt C_\ep|^2  =o(\lep),
\end{equation}
and hence  \eqref{cond2'} cannot be an equality either.
\end{proof}

 We now improve the comparison between $F_\ep $ and $\tilde{F_\ep}$ from the result of Theorem \ref{time_bound}.

\begin{cor}\label{energy_bound}
Suppose  the hypotheses of Theorem \ref{time_bound}.  Then as $\ep \rightarrow 0$, 
\begin{equation}\label{en_b_0}
 F_\ep(v_\ep,B_\ep)(t) < F_\ep(v_\ep,B_\ep)(0) + C_0 \ale + o(\ep^{1/8}) \text{ for all } t\in [0,T_0].
\end{equation}
\end{cor}
\begin{proof} This is a direct consequence of Lemma \ref{energy_comparison}, the well-preparedness and the bound on $\tilde{F}_\ep$ provided by Theorem \ref{time_bound}.

\end{proof}

\section{Convergence results}\label{sec5}

From here on we assume the various assumptions made in the introduction, in particular \eqref{well_prepared_def}, so that there are $n$ initial vortices of degrees $\pm 1$.  We now fix $C_0>0$ so that
\begin{equation}
 C_0 < \pi \left( \inf_{x\in\Omega} b(x) \right)
\end{equation}
and apply Theorem \ref{time_bound} to get $T_0>0$ so that the conclusions of the theorem hold on the interval $[0,T_0]$.  Then Corollary \ref{energy_bound} provides an estimate for $F_\ep(v_\ep,B_\ep)$ for all $t\in[0,T_0]$.

\subsection{Vortex trajectories}

Our first result allows us to define the  $n$ vortex paths in the time interval $[0,T_0]$.

\begin{lem}\label{vortex_path}
Let 
\begin{equation}
\gamma_0 = \min \{\abs{a_i(0) - a_j(0)} \;\vert\; i\neq j\} \cup \{\dist(a_i(0),\partial \Omega)\}. 
\end{equation}
Fix $0<\gamma_* < \gamma_0$ . Then there exists a $T_* = T_*(\gamma_*,C_0)$ with $T_*\in(0,T_0]$ so that
\begin{enumerate}
\item The space-time Jacobian $(\mu(v_\ep,B_\ep),V(v_\ep,B_\ep)) \rightarrow (\mu,V)$ in $(C^{0,1}(\Omega \times [0,T_0]))^*$.
\item There exist functions $a_i\in H^1([0,T_*];\Omega)$, $i=1,\dotsc,n$, so that for all  $t\in[0,T_*]$
\begin{equation}
\min \{\abs{a_i(t) - a_j(t)} \;\vert\; i\neq j\} \cup \{\dist(a_i(t),\partial \Omega)\} \ge \gamma_*, 
\end{equation}
\begin{equation}
\mu(t) = 2\pi \sum_{i=1}^n d_i(0) \delta_{a_i(t)},
\text{ and }
 V(t) = 2\pi \sum_{i=1}^n d_i(0) \dot{a}_i^\bot(t)  \delta_{a_i(t)}.
\end{equation}

\end{enumerate}
\end{lem}

\begin{proof}
The bounds provided by Theorem \ref{time_bound} and Corollary \ref{energy_bound} allow us to apply Proposition \ref{prod_est} to deduce the first item.  The second item would be standard if $b=1$ since then the quantization of the energy would prevent the nucleation of new vortices.  However, when $b \neq 1$ the energy is not quantized and so we must look for another structure to prove the result. 

We first recall that by the third item of Proposition \ref{prod_est}, the mapping $t \mapsto \langle \psi,\mu(t) \rangle$ is $H^1$ (and hence continuous) for any fixed $\psi \in C_c^\infty(\Omega)$.   Fix $0 < \delta < \pi (\inf b)- C_0$.  By the smoothness of $b$, there exists $\eta>0$ so that 
\begin{equation}\label{v_p_1}
 \abs{b(x) - b(a_i(0)) }\le \delta/(n\pi) \text{ for all } x\in \cup_{i=1}^n B(a_i(0),\eta).
\end{equation}

Standard ``Jacobian estimates'' (see \cite{js_jacob} or Chapter 6 of \cite{ss_book}) combined with the energy upper bound and the lower bound on $b$, 
give the quantized structure of the measure $\mu(t)$, i.e.
\begin{equation}
 \mu(t) = 2\pi \sum_{i=1}^{n(t)} d_i(t) \delta_{a_i(t)} 
\end{equation}
for $\abs{d_i(t)} \ge 1$.  Moreover, the $\Gamma-$convergence lower bound \eqref{gcv}, when combined with \eqref{en_b_0}, shows that
\begin{equation}\label{v_p_2}
\pi \sum_{i=1}^{n(t)} \abs{d_i(t)} b(a_i(t)) \le  \pi \sum_{i=1}^{n}  b(a_i(0)) + C_0.
\end{equation}

Let $\gamma_1 = \min\{(\gamma_0-\gamma_*)/2,\eta\}$.  For $i=1,\dotsc,n$ let $\psi_i \in C_c^\infty(\Omega)$ be a function supported in $ B(a_i(0),\gamma_1)$ so that $\psi_i(x) = 1$ for  $x\in B(a_i(0),\gamma_1/2)$.  Since $t \mapsto \langle \psi_i,\mu(t) \rangle$ is continuous and $\langle \psi_i,\mu(0)\rangle = 2 \pi d_i(0)$, we see that for $t\le T_*$ with $T_*$ sufficiently small at least one of the points $a_j(t)$, $j=1,\dots,n(t)$ must be contained in the ball $B(a_i(0),\gamma_1/2)$.  Up to relabeling the indices, we may assume that $a_i(t) \in B(a_i(0),\gamma_1/2)$ for $i=1,\dotsc,n$. This implies that  $n(t) \ge n$ for $t\le T_*$.  Now suppose by way of contradiction that $n(t) > n$ for some $t< T_*$.  Since $a_i(t)\in B(a_i(0),\eta)$ for $i=1,\dotsc,n$, the bound \eqref{v_p_1} implies that
\begin{equation}
\pi \abs{\sum_{i=1}^n b(a_i(t)) - b(a_i(0))  }\le \delta.
\end{equation}
Plugging this into \eqref{v_p_2}, we deduce the  bound
\begin{equation}
 C_0 + \delta \ge \pi \sum_{i=n+1}^{n(t)} \abs{d_i(t)} b(a_i(t)) \ge (n(t) -n)\pi (\inf b).
\end{equation}
The choice of $\delta$ and $C_0$  then implies that $n(t) - n < 1$, hence  $n(t) = n$ for $0 \le t \le T_*$.  The structure of $V$ follows from the structure of $\mu$ and the equation $\dt \mu + \curl{V} =0$.

Returning to the continuity of $t \mapsto \langle \psi_i,\mu(t)\rangle$, we find that $d_i(t) = d_i(0)$ for $i=1,\dotsc,n$.  Since $a_i(t) \in B(a_i(0),\gamma_1/2)$, we know that   
\begin{equation}
\min \{\abs{a_i(t) - a_j(t)} \;\vert\; i\neq j\} \cup \{\dist(a_i(t),\partial \Omega)\} \ge \gamma_*. 
\end{equation}
Finally, to see that $a_i\in H^1([0,T_*];\Omega)$ we use the fact that the mapping $t \mapsto \langle \psi,\mu(t)\rangle$ is in $H^1$ for the functions 
\begin{equation}
 \psi(x) = (x\cdot e) \psi_i(x), i=1,\dots,n
\end{equation}
for any unit vector $e\in \Rn{2}$.
\end{proof}

\begin{remark}
Since  the degree of the $i^{th}$ vortex does not change for any time in $[0,T_*]$, we may consolidate notation and write only $d_i$ in place of $d_i(0)$.
\end{remark}

\subsection{Normalized energy density}

We can now show that up to the extraction of a \emph{single} subsequence, the modified energy density converges  \emph{for all} $t\in[0,T_*]$.

\begin{prop}\label{density_convergence}
There exists a subsequence so that
\begin{equation}\label{den_con_0}
 \frac{\tilde{g}_\ep(v_\ep,B_\ep)(t)}{\ale} \wstar \nu(t) 
\end{equation}
weakly-$*$ in $(C^1(\Omega))^*$ for all $t\in[0,T_*]$.  For each $t \in [0,T_*]$,  $\nu(t)$ is a measure, and $\nu \in L^\infty([0,T_*]; \mathcal{M} (\Omega))$, where $\mathcal{M}(\Omega)$ is the space of Radon measures on $\Omega$.  Finally, 
\begin{equation}\label{den_con_1}
 \frac{\tilde{g}_\ep(v_\ep,B_\ep)}{\ale}  \to \nu(t) dt 
\end{equation}
weakly in the sense of measures on $\Omega \times [0,T_*].$
\end{prop}
\begin{proof}
The proof of \eqref{den_con_0} is the same as Proposition 4.4 of \cite{tice_2}.  The fact that $\nu(t)$ is a measure, the inclusion $\nu \in L^\infty([0,T_*];\mathcal{M}(\Omega))$, and the weak convergence \eqref{den_con_1} all follow easily from the energy bound of Corollary \ref{energy_bound} and weak compactness of measures in $\Omega$.
\end{proof}

Since we do not have precise control of the energy evolution yet, we cannot derive the exact structure of the measure $\nu$.  We can however, provide a lower bound that follows directly  from \eqref{gcv}.

\begin{lem}\label{nu_lower_bound}
 It holds that 
\begin{equation}
 \nu(t) \ge \frac{b}{2}  \abs{\mu(t)}.
\end{equation}
\end{lem}

\subsection{Stress-energy tensor, etc}

We define the stress-energy tensor associated to $(v_\ep,B_\ep)$ by 
\begin{multline}\label{T_def}
 T_\ep  = b \nab_{B_\ep} v_\ep \otimes \nab_{B_\ep} v_\ep 
- I_{2\times 2} \left( \frac{b}{2} \abs{\nab_{B_\ep} v_\ep }^2 + \frac{b^2}{4\ep^2}(1-\abs{v_\ep}^2)^2   - \hal \abs{\curl{B_\ep}}^2 \right) \\
+ I_{2\times 2} \left(\frac{(\abs{v_\ep}^2-1)}{2}(b f_\ep ) \right) .
\end{multline}
Note that 
\begin{equation}\label{Tadd}
T_\ep=  b \nab_{B_\ep} v_\ep \otimes \nab_{B_\ep} v_\ep  - I_{2\times 2}\(\tilde{g}_\ep(v_\ep, B_\ep)  - \abs{\curl{B_\ep}}^2 \).\end{equation}

The next result provides for the convergence of the (normalized) stress-energy tensor along with a couple other quantities.

\begin{lem}\label{measure_converge}
Let $T_\ep$ be given by \eqref{T_def}.  Then up to the extraction of a subsequence, as $\ep\to 0$ we have
\begin{equation}
 \frac{T_\ep}{\ale} \rightarrow T,
\end{equation}
\begin{equation}
 \frac{b (\dt v_\ep,\nab_{B_\ep} v_\ep) }{\ale} \rightarrow p,
\end{equation}
and 
\begin{equation}
\alpha \frac{b \abs{\dt v_\ep}^2}{\ale} + \sigma \frac{\abs{\dt B_\ep}^2}{\ale} \rightarrow \zeta
\end{equation}
in the weak sense of measures on $\Omega \times [0,T_*]$.  
\end{lem}
\begin{proof}
All the terms defining $T_\ep$ are seen to be controlled by the energy density $\tilde{g}_\ep$. Therefore,
Cauchy's inequality and the energy bounds of Theorem \ref{time_bound} and Corollary \ref{energy_bound} show that
\begin{equation}
 \int_0^{T_*} \int_\Omega \frac{\abs{T_\ep}}{\ale} +\frac{b (\dt v_\ep,\nab_{B_\ep} v_\ep) }{\ale}  +\alpha \frac{b \abs{\dt v_\ep}^2}{\ale} + \sigma \frac{\abs{\dt B_\ep}^2}{\ale} \le C .
\end{equation}
Compactness in the sense of measures follows directly from these bounds.
\end{proof}

If we knew more information about the evolution of the modified energy $\tilde{F}_\ep$,  we could say more about the structure of the limiting measures $T,$ $p,$ and $\zeta$.  Instead we will resort to a Lebesgue decomposition of all the measures with respect to the vorticity measure $\mu$, which we recall is equal to $2\pi \sum_{i=1}^n d_i \delta_{a_i(t)}dt$.

\begin{lem}\label{lebesgue_decomp}
We can decompose
\begin{equation}\label{ld_02}
 \nu  = \nu_0(t) dt + \sum_{i=1}^n \nu_i(t) \delta_{a_i(t)} dt, \;
 T = T_0  + \sum_{i=1}^n T_i(t) \delta_{a_i(t)} dt,
\end{equation}
\begin{equation}\label{ld_03}
\zeta  = \zeta_0  + \sum_{i=1}^n \zeta_i(t) \delta_{a_i(t)} dt, \text{ and } p  = p_0  + \sum_{i=1}^n p_i(t) \delta_{a_i(t)} dt,
\end{equation}
where $T_0$, $\zeta_0$, and $\abs{p_0}$ are mutually singular with $\mu$ as measures on $\Omega \times [0,T_*]$ and $\nu_0(t)$ is mutually singular with $\mu(t)$ as measures on $\Omega$ for a.e. $t\in [0,T_*]$.  We have
\begin{equation}\label{ld_04}
  \nu_0  \in L^\infty([0,T_*];\mathcal{M}(\Omega)), 
\end{equation}
where $\mathcal{M}(\Omega)$ is the space of Radon measures on $\Omega$.   For $i=1,\dotsc,n$ we also have that 
\begin{equation}\label{ld_06}
 \nu_i \in L^\infty([0,T_*]; \Rn{}).
\end{equation}
Finally, we have that $\nu_0, \zeta_0, \zeta_i(t) \ge 0$, and the quantities $\nu_i(t)$ obey the bounds $\pi b(a_i(t)) \le \nu_i(t) \le \pi n \pnorm{b}{\infty} +C_0$ for $i=1,\dotsc,n$, where $C_0>0$ is the constant chosen at the beginning of Section \ref{sec5}.
\end{lem}

\begin{proof}
We perform a Lebesgue decomposition of the measures $\nu$, $T$, $\zeta$, and $p$ with respect to $\mu$ in order to write 
\begin{equation}\label{ld_1}
 \nu = \nu_0 + \nu_\mu,\; T = T_0 + T_\mu,\; \zeta = \zeta_0 + \zeta_\mu, \text{ and }p = p_0 + p_\mu, 
\end{equation}
where $\nu_0,$ $\abs{T_0},$ $\zeta_0,$ and $\abs{p_0}$ are mutually singular with $\mu$ as measures on $\Omega \times [0,T_*]$ and $\nu_\mu$, $T_\mu$, $\zeta_\mu$, and $p_\mu$ are absolutely continuous with respect to $\mu$.  According to Lemma \ref{vortex_path}, we may write $\mu = 2\pi \sum_{i=1}^n d_i \delta_{a_i(t)} dt$, so that $\mu$ is supported along  the curves $\{(a_i(t),t) \;\vert\; t \in [0,T_*]\}$, which are parameterized by $H^1$ functions.  This allows us to further decompose
\begin{equation}\label{ld_2}
 \nu_\mu = \sum_{i=1}^n \nu_i(t) \delta_{a_i(t)} dt, \; T_\mu = \sum_{i=1}^n T_i(t) \delta_{a_i(t)} dt,
\end{equation}
and
\begin{equation}\label{ld_3}
 \zeta_\mu = \sum_{i=1}^n \zeta_i(t) \delta_{a_i(t)} dt, \; p_\mu = \sum_{i=1}^n p_i(t) \delta_{a_i(t)} dt,
\end{equation}
where  $\nu_i$, $T_i$, $\zeta_i$, $p_i$ are the Radon-Nikodym derivatives of $\nu_\mu$, $T_\mu$, $\zeta_\mu$, $p_\mu$ along the curve $\{(a_i(t),t) \;\vert\; t \in [0,T]\}$.  The decompositions \eqref{ld_1} and \eqref{ld_3} imply \eqref{ld_03} and the $T$ decomposition in \eqref{ld_02}.

To finish the $\nu$ decomposition in \eqref{ld_02} we must extract more structure from $\nu_0$.  Proposition \ref{density_convergence}  implies  $\nu = \nu(t) dt$, where the mapping $t \mapsto \nu(t)$  is in $L^\infty([0,T_*];\mathcal{M}(\Omega))$.  Using this and the above decomposition, we know that $\nu_0  \in L^\infty([0,T_*];\mathcal{M}(\Omega))$ and $\nu_i  \in L^\infty([0,T_*];\Rn{})$ for $i=1,\dotsc,n$.  This implies the decomposition of $\nu$ in \eqref{ld_02} as well as the $\nu$ inclusions in \eqref{ld_04} and \eqref{ld_06}.

The bounds  $\nu_0(t), \zeta_0(t), \zeta_i(t) \ge 0$ are trivial.  The lower bounds $\nu_i(t) \ge \pi b(a_i(t))$ for $i=1,\dotsc,n$ follow from Lemma \ref{nu_lower_bound}.  The upper bounds $\nu_i(t) \le \pi n \pnorm{b}{\infty} + C_0$ for $i=1,\dotsc,n$ follow from the energy upper bounds of Corollary \ref{energy_bound} and the well-preparedness assumption \eqref{well_prepared_def}.  
\end{proof}

We now prove some estimates of $\abs{p_0}$, $\abs{T_0}$, and $\abs{\sigma \curl{B_*} \dt B_*^\bot}$ in terms of $\zeta_0$ and $\nu_0$.

\begin{lem}\label{decomp_estimates}
 It holds that 
\begin{equation}\label{dece_01}
 \abs{T_0} \le C \nu_0(t)dt.
\end{equation}
Let $\eta \in C^0([0,T_*];\Rn{})$ satisfy $\eta(t) > 0$ for all $t \in [0,T_*]$.  Then we may estimate
\begin{equation}\label{dece_02}
 \abs{p_0} \le \eta \frac{\zeta_0}{2 \alpha} + \frac{\nu_0(t) dt}{\eta}
\end{equation}
as well as
\begin{equation}\label{dece_03}
 \abs{\sigma \curl{B_*} \dt B_*^\bot} \le \eta \frac{\zeta_0}{2} + \frac{\sigma \nu_0(t) dt}{ \eta}.
\end{equation}
\end{lem}

\begin{proof}
From \eqref{Tadd} we see that $\abs{T_\ep}/\ale  \le C  \tilde{g}_\ep(v_\ep,B_\ep)/\ale$, and passing to the limit reveals that $\abs{T} \le C \nu$.  The decomposition of $T$  provided by Lemma \ref{lebesgue_decomp} implies that $\abs{T} = \abs{T_0} + \sum_{i=1}^n \abs{T_i(t)} \delta_{a_i(t)}dt$, with $\abs{T_0}$ mutually singular with  $\mu$, and hence also with   $\sum_{i=1}^n \abs{T_i(t)}\delta_{a_i(t)}dt$ and $\sum_{i=1}^n \nu_i(t) \delta_{a_i(t)}dt$.  The estimate  \eqref{dece_01} then follows.

Now let $\eta \in C^0([0,T_*];\Rn{})$ satisfy $\eta(t) > 0$.  The Cauchy-Schwarz inequality allows us to bound
\begin{equation}
 \frac{ \abs{  b (\dt v_\ep,\nab_{B_\ep} v_\ep) }}{\ale} \le  \frac{\eta}{2} \frac{b \abs{\dt v_\ep}^2}{ \ale} + \frac{1}{\eta} \frac{b\abs{\nab_{B_\ep}v_\ep}^2}{2\ale}.
\end{equation}
Passing to the limit, we find that 
\begin{equation}
 \abs{p} \le \frac{\eta \zeta}{2\alpha} + \frac{\nu}{\eta}.
\end{equation}
The bound \eqref{dece_02} follows from this, the decompositions of $\zeta$ and $\nu$ provided by Lemma \ref{lebesgue_decomp}, and the fact that $\abs{p} = \abs{p_0} + \sum_{i=1}^n \abs{p_i} \delta_{a_i }$ with $\abs{p_0}$ mutually singular with  $\sum_{i=1}^n \zeta_i \delta_{a_i}$ and $\sum_{i=1}^n \nu_i  \delta_{a_i}$. 

Similarly, we may bound
\begin{equation}
\frac{ \abs{\sigma \curl{B_\ep} \dt B_\ep^\bot} }{\ale} \le \frac{\eta}{2} \frac{\sigma \abs{\dt B_\ep}^2}{\ale} + \frac{\sigma}{\eta} \frac{\abs{\curl{B_\ep}}^2}{2\ale},
\end{equation}
which implies, upon passing to the limit, that as measures
\begin{equation}
  \abs{\sigma \curl{B_*} \dt B_*^\bot} \le \frac{\eta \zeta}{2} + \frac{\sigma \nu}{\eta}. 
\end{equation}
Then \eqref{dece_03} follows from the decompositions of Lemma \ref{lebesgue_decomp} and \eqref{b_l_04} of Proposition \ref{B_limits}, which implies that $\abs{\sigma \curl{B_*} \dt B_*^\bot}$ is mutually singular with $\mu$.
\end{proof}

We now compute the divergence of $T_\ep$  and relate its limit as $\ep \rightarrow 0$ to the Lebesgue decomposition of $p,\nu, T$.

\begin{lem}\label{tensor_diverge}
 The stress-energy tensor satisfies
\begin{multline}\label{t_d_01}
 \diverge{T_\ep} = \alpha b(\dt v_\ep,\nab_{B_\ep} v_\ep) - \sigma \curl B_\ep\,  \dt B_\ep^\bot - \frac{\beta \ale b}{2} V(v_\ep,B_\ep)  -  \mu(v_\ep,B_\ep) b Z_\ep^\bot   \\
 -\left(\tilde{g}_\ep(v_\ep,B_\ep) - \frac{\text{tr}(T_\ep)}{2} \right) \nab \log{b} 
+ \frac{(\abs{v_\ep}^2-1)}{2} \left[ \beta \ale  b \dt B_\ep  +  b \nab f_\ep   - f_\ep \nab b \right].
\end{multline}
Dividing by $\ale$ and passing to the limit $\ep \rightarrow 0$, we have that
\begin{equation}\label{t_d_02}
 \diverge{T} = \alpha p - \frac{\beta b}{2} V - \mu b Z^\bot - \left(\nu -\frac{\text{tr}(T)}{2}\right) \nab \log{b} - \sigma \curl{B_*} \dt B_*^\bot,
\end{equation}
where $B_*$ is the vector field given by Proposition \ref{B_limits}.  We have the decomposition 
\begin{equation}\label{t_d_05}
 \diverge{T}  = S_0  + \sum_{i=1}^n S_i(t) \delta_{a_i(t)} dt,
\end{equation}
where   
\begin{equation}\label{t_d_03}
 S_0 = \alpha p_0 -\left( \nu_0 - \frac{\text{tr}(T_0)}{2} \right) \nab \log{b}  - \sigma \curl{B_*} \, \dt B_*^\bot
\end{equation}
with $\abs{S_0}$  mutually singular with $\mu$, and
\begin{equation}\label{t_d_04}
 S_i =   \alpha p_i -  \beta  \pi d_i b(a_i) \dot{a}_i^\bot - 2\pi d_i b(a_i) Z^\bot(a_i) - \left(\nu_i - \frac{\text{tr}(T_i)}{2} \right)\nab \log{b}(a_i).
\end{equation}
\end{lem}
\begin{proof}
A direct calculation reveals that
\begin{multline}\label{t_d_1}
 \diverge{T_\ep} =  ( b \Delta_{B_\ep} v_\ep + b^2 v_\ep \ep^{-2}(1-\abs{v_\ep}^2) + \nab_{B_\ep}v_\ep \cdot \nab b  , \nab_{B_\ep} v_\ep)  - h_\ep' (b (i v_\ep,\nab_{B_\ep} v_\ep) + \nab^\bot h_\ep')^\bot
\\
- \nab{b} \frac{\abs{\nab_{B_\ep} v_\ep}^2}{2} - \nab {b^2} \frac{(1-\abs{v_\ep}^2)^2}{4\ep^2}  
+ \nab \left( \frac{(\abs{v_\ep}^2-1)}{2} \right)   \left(b  f_\ep \right)     +   \frac{(\abs{v_\ep}^2-1)}{2}     \left( f_\ep\nab b  + b \nab f_\ep \right)  .
\end{multline}
By plugging in the equations \eqref{s_r_00}--\eqref{s_r_01} and using \eqref{m_e_e_1}  we may rewrite
\begin{multline}\label{t_d_2}
 ( b \Delta_{B_\ep} v_\ep + b^2 v_\ep \ep^{-2}(1-\abs{v_\ep}^2) + \nab_{B_\ep}v_\ep \cdot \nab b  , \nab_{B_\ep} v_\ep)  - h_\ep' (b (i v_\ep,\nab_{B_\ep} v_\ep) + \nab^\bot h_\ep')^\bot \\
= \alpha b( \dt v_\ep,\nab_{B_\ep} v_\ep  ) - \frac{\beta \ale b}{2} V(v_\ep,B_\ep) - \sigma h_\ep' \dt B_\ep^\bot - \mu(v_\ep,B_\ep) b Z_\ep^\bot \\
-\beta \ale \frac{\dt B_\ep}{2}(1-\abs{v_\ep}^2)
- \nab \left(\frac{(\abs{v_\ep}^2-1)}{2} \right) \left(b  f_\ep \right),
\end{multline}
where we have employed the identities
\begin{equation}
 (-2i \nab_{B_\ep} v_\ep \cdot b Z_\ep,\nab_{B_\ep} v_\ep) = 2 (\partial_1^{B_\ep} v_\ep, i \partial_2^{B_\ep} v_\ep) b Z_\ep^\bot
\end{equation}
and 
\begin{equation}
 (\abs{v_\ep}^2-1) h'_\ep + 2 (\partial_1^{B_\ep} v_\ep, i \partial_2^{B_\ep} v_\ep)  = - \curl(iv_\ep,\nab_{B_\ep} v_\ep) -h'_\ep = -\mu(v_\ep,B_\ep)
\end{equation}
to identify the $- \mu(v_\ep,B_\ep) b Z_\ep^\bot$ term in \eqref{t_d_2}.  From \eqref{T_def} we may  calculate
\begin{equation}\label{t_d_10}
 \text{tr}(T_\ep) = -b^2\frac{(1-\abs{v_\ep}^2)^2}{2 \ep^2}  + \abs{h_\ep'}^2 + (\abs{v_\ep}^2-1)b f_\ep,
\end{equation}
which implies
\begin{multline}\label{t_d_3}
 - \nab{b} \frac{\abs{\nab_{B_\ep} v_\ep}^2}{2} - \nab {b^2} \frac{(1-\abs{v_\ep}^2)^2}{4\ep^2} = -g_\ep(v_\ep,B_\ep) \nab \log b \\
+  \nab \log b \left(  -b^2\frac{(1-\abs{v_\ep}^2)^2}{4 \ep^2}  + \frac{\abs{h_\ep'}^2}{2}  \right) 
= \nab \log b \left(  -\tilde{g}_\ep(v_\ep,B_\ep)  + \frac{\text{tr}(T_\ep)}{2} - (\abs{v_\ep}^2-1) b f_\ep     \right)  \\
= \nab \log b \left(  -\tilde{g}_\ep(v_\ep,B_\ep)  + \frac{\text{tr}(T_\ep)}{2}\right)  - (\abs{v_\ep}^2-1)  f_\ep \nab b  .
\end{multline}
Combining \eqref{t_d_1}--\eqref{t_d_2} with \eqref{t_d_3} then yields \eqref{t_d_01}.  Equation \eqref{t_d_02} follows by dividing \eqref{t_d_01} by $\ale$ and passing to the limit $\ep \rightarrow 0$, using the fact that $\ale^2 (1-\abs{v_\ep}^2) \to 0$ (by the upper bound on the energy) as well as  Lemma \ref{measure_converge} and  Proposition \ref{B_limits} to identify the structure of the limit.   The decomposition \eqref{t_d_05} with $S_0$ given by \eqref{t_d_03} and $S_i$ given by \eqref{t_d_04} follows by decomposing the terms on the right side of \eqref{t_d_02} according to Lemmas \ref{vortex_path} and \ref{lebesgue_decomp}, noting that \eqref{b_l_04} of Proposition \ref{B_limits} implies that $\sigma \curl{B_*} \dt B_*^\bot$ is absolutely continuous with respect to Lebesgue measure on $\Omega \times [0,T_*]$, and hence mutually singular with $\mu$.
\end{proof}

Now we show that $\abs{S_0}$ can be controlled in terms of $\nu_0(t)dt$ and $\zeta_0$.

\begin{lem}\label{S0_estimate}
Let $S_0$  be as in Lemma \ref{tensor_diverge}. For any $\eta\in C^0([0,T_*];\Rn{})$ satisfying $\eta(t) >0$,  we have the estimate
\begin{equation}\label{s_e_01}
 \abs{S_0} \le \eta \zeta_0 + \left( C + \frac{\alpha + \sigma}{\eta} \right)\nu_0(t) dt.
\end{equation}
\end{lem}

\begin{proof}
Using the definition of $S_0$ from Lemma \ref{tensor_diverge} along with estimates of Lemma  \ref{decomp_estimates}, we may bound
\begin{multline}
 \abs{S_0} \le  \alpha \abs{p_0} +\left( \nu_0 +  \abs{T_0} \right) \nab \log{b}  +\abs{ \sigma \curl{B_*} \dt B_*^\bot} \\
\le \frac{\eta \zeta_0}{2} + \frac{\alpha \nu_0}{\eta} + C \nu_0 + \frac{\eta \zeta_0}{2} + \frac{\sigma \nu_0}{\eta} = \eta \zeta_0 + \left( C + \frac{\alpha + \sigma}{\eta} \right)\nu_0,
\end{multline}
which is \eqref{s_e_01}.  Here we have used the fact that $b$ is smooth and bounded below, which is guaranteed by \eqref{b_lower_bound}.
\end{proof}

\section{Dynamics}\label{sec6}

We now use the convergence results of the last section to derive the dynamics of the vortices.  We begin with a result that allows us to relate $p_i$ to  the vortex velocity, $\dot{a}_i$.  This result does not constitute the full dynamical law for $a_i$ since we do not yet know the value of $p_i$ or $\nu_i$.

\begin{prop}\label{velocities}
For $i=1,\dotsc,n$, it holds that 
\begin{equation}
  p_i(t) = -\nu_i(t) \dot{a}_i(t) 
\end{equation}
for a.e. $t\in[0,T_*]$.

\end{prop}
\begin{proof}
Fix $0 < \eta < \gamma_*$ (with $\gamma_*$ given by Lemma \ref{vortex_path}) and a smooth vector field $Y:[0,T_*]\rightarrow \Rn{2}$.  Let $\phi$ be a smooth function with support in $B(0,1)\times [0, T_*]$ so that $\nab \phi(0,t) = Y(t)$ for all $t \in [0, T_*]$.  Fix $i\in\{1,\dotsc,n\}$  and let $a^\gamma_i:[0,T_*]\rightarrow \Omega$ be a smooth mollification of the path $a_i(t)$, with $\gamma>0$ the mollification parameter.  Define  $\psi_\gamma(x,t) = \eta \phi((x-a^\gamma_i(t))/\eta,t)$ and  $\psi(x,t) = \eta \phi((x-a_i(t))/\eta,t)$.

Then the energy evolution equation \eqref{m_e_e_01} implies that
\begin{multline}
 \int_\Omega \psi_\gamma \frac{\tilde{g}_\ep(T_*)}{\ale} +   
\int_{\partial \Omega} \psi_\gamma \frac{(\abs{v_\ep(T_*)}^2-1)}{4\ale} \nab b \cdot \nu
  - \int_\Omega \psi_\gamma \frac{\tilde{g}_\ep(0)}{\ale}
- \int_{\partial \Omega} \psi_\gamma \frac{(\abs{v_\ep(0)}^2-1)}{4\ale} \nab b \cdot \nu \\
+ \int_0^{T_*} \int_\Omega \psi_\gamma \left( \alpha \frac{b\abs{\dt v_\ep}^2}{\ale} + 
  \sigma \frac{\abs{\dt B_\ep}^2}{\ale} \right) 
= \int_0^{T_*} \int_\Omega -\nab \psi_\gamma \cdot \frac{b (\dt v_\ep,\nab_{B_\ep} v_\ep) + h_\ep' \dt B_\ep^\bot}{\ale} \\
+ \int_0^{T_*} \int_\Omega \psi_\gamma V(v_\ep,B_\ep) \cdot b Z  + \int_0^{T_*} \int_\Omega \dt \psi_\gamma  \frac{\tilde{g}_\ep}{\ale} .
\end{multline}
We may pass to the limit $\ep \rightarrow 0$ in the last equation by using Propositions  \ref{B_bound}, \ref{B_limits}, and  \ref{density_convergence} along with Lemmas  \ref{vortex_path}, \ref{measure_converge}, to get
\begin{multline}
 \int_\Omega \psi_\gamma \nu(T_*) - \int_\Omega \psi_\gamma  \nu(0) + \int_0^{T_*} \int_\Omega \psi_\gamma  \zeta = \int_0^{T_*} \int_\Omega -\nab \psi_\gamma  \cdot (p + \curl{B_*} \dt B_*^\bot) \\
+ \int_0^{T_*} \int_\Omega \psi_\gamma  V \cdot Z  + \int_0^{T_*} \int_\Omega \dt \psi_\gamma  \, \nu.
\end{multline}
Now,
$\dt \psi_\gamma(x,t)= \eta \dt \phi((x-a_i^\gamma(t))/\eta, t)- \dot{a}_i^\gamma \cdot  \nab 
\phi(( x-a_i^\gamma(t))/\eta, t)$. Letting $\gamma
\rightarrow 0$, using the boundedness of all of the measures involved given by Lemma \ref{lebesgue_decomp}, and employing dominated convergence, we deduce that
\begin{multline}\label{vel_1}
 \int_\Omega \psi \nu(T_*) - \int_\Omega \psi  \nu(0) + \int_0^{T_*} \int_\Omega \psi  \zeta = \int_0^{T_*} \int_\Omega -\nab \psi  \cdot (p + \curl{B_*} \dt B_*^\bot )\\
+ \int_0^{T_*} \int_\Omega \psi  V \cdot Z  + \int_0^{T_*} \int_\Omega \(\eta \dt \phi((x-a_i(t))/\eta,t)  - \dot{a}_i (t) \cdot       \nab \phi((x-a_i(t))/\eta,t) 
 \) \nu.
\end{multline}
Note that here $\dot{a}_i$ is the time derivative of $a_i\in H^1([0, T_*])$ hence is an $L^2([0,T_*]; \Rn{2})$ function, while $\nu $ is $L^\infty([0,T_*],\mathcal{M}(\om))$.  Hence the product $\dot{a}_i\nu$ makes sense as an element of $L^2([0,T_*];  \Rn{2} \otimes \mathcal{M}(\om))$.

Recall that $\pnorm{\psi}{\infty} + \pnorm{\dt \psi}{\infty} \le C \eta$ and that the support of $\psi$ lies in a $\eta-$neighborhood of the path $a_i(t)$.  Passing to the limit $\eta \rightarrow 0$  reveals that 
\begin{equation}
  \int_\Omega \psi \nu(T_*) - \int_\Omega \psi  \nu(0) + \int_0^{T_*} \int_\Omega \psi  \zeta  \rightarrow 0 
\end{equation}
and
\begin{equation}
\int_0^{T_*} \int_\Omega \psi  V \cdot Z +    \eta \dt \phi((x-a_i(t))/\eta,t)      \nu \rightarrow 0. 
\end{equation}
Since $\nab \psi$ is supported in  $B(0,\eta)$ and $\pnorm{\nab \psi}{\infty} \le C < \infty$, we also have that
\begin{equation}
 \int_0^{T_*} \int_\Omega -\nab \psi  \cdot  \curl{B_*}\,  \dt B_*^\bot  \rightarrow 0.
\end{equation}
We may then pass to the limit $\eta \rightarrow 0$ in \eqref{vel_1} and utilize Lemma \ref{lebesgue_decomp}, and the fact that $\nab \phi(0,t)= \nab \psi(a_i(t),t) = Y(t)$ to deduce that
\begin{equation}
 0 =  \int_0^{T_*} Y \cdot (p_i + \nu_i \dot{a}_i).
\end{equation}
This result holds for any choice of $Y\in C^\infty([0,T_*];\Rn{2})$, which implies that
\begin{equation}
  p_i = -\nu_i \dot{a}_i 
\end{equation}
for a.e. $t \in [0,T_*]$.  This is the desired result.

\end{proof}

Now we can use a similar argument to show that $T_i(t)=0$ for a.e. $t \in [0,T_*]$.

\begin{prop}\label{tensor_structure}
For a.e. $t \in [0,T_*]$ and $i=1,\dotsc,n$ it holds that $T_i(t) =0$.
\end{prop}
\begin{proof}
Let $\psi \in C_c^\infty(\Rn{2};\Rn{})$ be such that $\supp(\psi) \in B(0,\gamma_*)$, where $\gamma_*$ is given by Lemma \ref{vortex_path}, and so that $\psi(x)=1$ on $B(0,\gamma_*/2)$.  Let $K_i \in C^0([0,T_*];\Rn{2\times 2})$ for $i=1,\dotsc,n$.  Define the vector field 
\begin{equation}
 \Xi(x,t) = \eta \sum_{i=1}^n K_i(t) \cdot \frac{(x-a_i(t))}{\eta} \psi \left(\frac{x-a_i(t)}{\eta}\right).
\end{equation}
Since Lemma \ref{vortex_path} says that $a_i \in H^1 \hookrightarrow C^{0,1/2}$, we have that $\Xi \in C^0(\Omega \times [0,T_*]; \Rn{2})$ and $D \Xi \in C^0(\Omega \times [0,T_*]; \Rn{2\times 2})$.  Note that, unlike in Proposition \ref{velocities}, we do not need continuity of $\dt \Xi$, so we do not have to use a smoothing of the vortex paths.

According to the divergence theorem, we have
\begin{equation}\label{t_s_1}
 \int_0^{T_*}\int_\Omega -\frac{T_\ep}{\ale}:D \Xi = \int_0^{T_*} \int_\Omega \frac{\diverge{T_\ep}}{\ale}\cdot \Xi.
\end{equation}
By the above continuity results, we may pass to the limit in \eqref{t_s_1} and employ the decompositions of $T$ and $\diverge{T}$, given respectively by Lemmas \ref{lebesgue_decomp} and \ref{tensor_diverge}, to see that
\begin{multline}\label{t_s_2}
  -\int_0^{T_*}\int_\Omega T_0 : D \Xi -\int_0^{T_*} \sum_{i=1}^n T_i(t): D \Xi(a_i(t),t) \\
=  \int_0^{T_*}\int_\Omega S_0 \cdot \Xi + \int_0^{T_*} \sum_{i=1}^n S_i(t) \cdot \Xi(a_i(t),t).
\end{multline}
By construction $\Xi(a_i(t),t) = 0$ and $D \Xi(a_i(t),t) = K_i(t)$, so \eqref{t_s_2} becomes
\begin{equation}\label{t_s_3}
  \int_0^{T_*} \sum_{i=1}^n T_i(t): K_i(t) = -\int_0^{T_*}\int_\Omega T_0 : D \Xi + S_0 \cdot \Xi.
\end{equation}

The vector field $\Xi$ satisfies $\pnormspace{\Xi}{\infty}{\Omega \times [0,T_*]} \le C \eta$.   This implies that
\begin{equation}
 \abs{\int_0^{T_*}\int_\Omega S_0 \cdot \Xi } \le C \eta \int_0^{T_*}\int_\Omega \abs{S_0}  \to 0 \text{ as }\eta \to 0
\end{equation}
since $\abs{S_0}$ has finite mass.  Also, since $T_0$ is singular with respect to $\mu$ and $D \Xi$ satisfies  $\pnormspace{D\Xi}{\infty}{\Omega \times [0,T_*]} \le C$  and $\supp(D \Xi(\cdot,t))\subset \cup_{i=1}^n B(a_i(t),\eta)$, we have that
\begin{equation}
 -\int_0^{T_*}\int_\Omega T_0 : D \Xi \to 0  \text{ as } \eta \to 0.
\end{equation}
Hence, taking the limit $\eta \to 0$ in \eqref{t_s_3}, we find that
\begin{equation}
 \int_0^{T_*} \sum_{i=1}^n T_i(t): K_i(t)=0.
\end{equation}
Since the $K_i$ were arbitrary, we immediately deduce that $T_i(t) = 0$ for a.e. $t \in [0,T_*]$. 
\end{proof}

Now we can deduce some estimates for $S_i$, as defined in Lemma \ref{tensor_diverge}.

\begin{lem}\label{Si_T_estimate}
Let  $S_i$, $i=1,\dotsc,n$ be as in Lemmas \ref{tensor_diverge}.  Then $S_i \in L^2([0,T_*];\Rn{2})$ for each $i=1,\dotsc,n$.  Moreover, if $Y_i \in L^2([0,T_*];\Rn{2})$ for each $i=1,\dotsc,n$, then for any $t \in [0,T_*]$ we have the estimate 
\begin{equation}\label{si_e_0}
 \int_0^{t} \sum_{i=1}^n S_i \cdot Y_i  \le  \hal \int_0^{t} \int_{\Omega}  \zeta_0 +  \int_0^t C\left( 1 +  \sum_{i=1}^n \abs{Y_i}^2\right)  \int_\Omega  \nu_0
\end{equation}
for some $C>0$.
\end{lem}

\begin{proof}
The result of Proposition \ref{velocities}, when combined with the definition of $S_i$ given by \eqref{t_d_04} in Lemma \ref{tensor_diverge} and the vanishing of $T_i$ given by Proposition \ref{tensor_structure}, implies that
\begin{equation}
  S_i =   - \alpha \nu_i \dot{a}_i -  \beta  \pi d_i b(a_i) \dot{a}_i^\bot - 2\pi d_i b(a_i) Z^\bot(a_i) - \nu_i \nab \log{b}(a_i).
\end{equation}
Lemma \ref{lebesgue_decomp} implies that $\nu_i \in L^\infty([0,T_*])$, Remark \ref{Z_smooth} shows that $Z$ is bounded, and assumption  \eqref{b_lower_bound} provides the boundedness of $b$ and $\nab \log{b}$; then since we know from Lemma \ref{vortex_path} that $\dot{a}_i \in L^2([0,T_*]; \Rn{2})$, we find that $S_i \in L^2([0,T_*]; \Rn{2})$.

We now turn to the proof of \eqref{si_e_0}, assuming initially that $Y_i \in C^0([0,T_*];\Rn{2})$ for each $i=1,\dotsc,n$.  Let $\psi \in C_c^\infty(\Rn{2};\Rn{})$ be such that $\supp(\psi) \in B(0,\gamma_*)$, where $\gamma_*$ is given by Lemma \ref{vortex_path}, and so that $\psi(x)=1$ on $B(0,\gamma_*/2)$.  Define the vector field 
\begin{equation}
 \Xi(x,t) = \sum_{i=1}^n Y_i(t) \psi(x-a_i(t)).
\end{equation}
For each fixed $t$ we have $\supp(\Xi(\cdot,t)) \subset \cup_{i=1}^n B(a_i(t),\gamma_*)$, and the choice of $\gamma_*$ implies that $a_j(t) \notin B(a_i(t),\gamma_*)$ for $i\neq j$.  Since $a_i \in C^{0,1/2}$ by Lemma \ref{vortex_path}, we know that $\Xi \in C^0(\Omega \times [0,T_*] ; \Rn{2})$ and  $D\Xi \in C^0(\Omega \times [0,T_*] ; \Rn{2\times 2})$.  If we define the continuous function
\begin{equation}
 M(t) := \sum_{i=1}^n \abs{Y_i(t)}^2,
\end{equation}
then for each  $t \in [0,T_*]$ we may bound
\begin{equation}\label{si_e_1}
 \norm{\Xi(\cdot,t)}_{C^1(\Omega)} \le C_* \sqrt{M(t)} 
\end{equation}
for a constant $C_*$ depending on $\gamma_*$ and $n$ but not on $t$.

For  $t \in [0,T_*]$ we may argue as in \eqref{t_s_1}--\eqref{t_s_2} of Proposition \ref{tensor_structure}, replacing the temporal integration interval $[0,T_*]$ with $[0,t]$ to find that
\begin{equation}\label{si_e_2}
 -\int_0^t \int_\Omega T_0 : D \Xi - \int_0^t \sum_{i=1}^n T_i : D \Xi(a_i) = \int_0^t \int_\Omega S_0 \cdot \Xi  + \int_0^t \sum_{i=1}^n S_i \cdot \Xi(a_i).
\end{equation}
We have $\Xi(a_i(r),r) = Y_i$ and $D\Xi(a_i(r),r) = 0$ for $i=1,\dotsc,n$ and $r \in [0,t]$, so \eqref{si_e_2} becomes
\begin{equation}\label{si_e_3}
\int_0^t \sum_{i=1}^n S_i \cdot Y_i =  -\int_0^t \int_\Omega T_0 : D \Xi + S_0 \cdot \Xi.
\end{equation}
But by \eqref{si_e_1} we may estimate
\begin{multline}\label{si_e_4}
 -\int_0^t \int_\Omega T_0 : D \Xi + S_0 \cdot \Xi  \le \int_0^t \int_\Omega \norm{\Xi(\cdot,r)}_{C^1(\Omega)}  \left(\abs{T_0} + \abs{S_0}\right)  \\
\le C_*   \int_0^t   \int_\Omega \sqrt{M} \left( \abs{T_0} + \abs{S_0}\right).
\end{multline}

Now we set $\eta \in C^0$ according to
\begin{equation}
\eta(s) =  \frac{1}{2 C_* (1+ \sqrt{M(s)})}
\end{equation}
with $C_*$ the constant on the right side of \eqref{si_e_1};  this choice  implies that
\begin{equation}
  \eta C_* \sqrt{M} \le  \frac{1}{2}, \text{ and }
\end{equation}
\begin{equation}
C_* \sqrt{M} \left( C + \frac{\alpha + \sigma}{\eta} \right) = C_* \sqrt{M}\left( C + 2 C_* (\alpha + \sigma) (1+ \sqrt{M(s)})   \right) \le C(1+M),
\end{equation}
where we have used Cauchy's inequality, and in the last bound $C$ depends on $C_*$.  We use this $\eta$ in Lemma \ref{S0_estimate} to bound
\begin{equation}\label{si_e_5}
 C_* \sqrt{M} \abs{S_0} \le \eta C_* \sqrt{M} \zeta_0 +  C_* \sqrt{M} \left( C + \frac{\alpha + \sigma}{\eta} \right) \nu_0(s) ds \le \frac{\zeta_0}{2} + C(1+M(s)) \nu_0(s) ds.
\end{equation}
On the other hand, \eqref{dece_01} of Lemma \ref{decomp_estimates} and Cauchy's inequality imply that
\begin{equation}\label{si_e_6}
C_* \sqrt{M} \abs{T_0} \le C(1+M(s)) \nu_0(s) ds. 
\end{equation}

Now we use \eqref{si_e_5}--\eqref{si_e_6} in \eqref{si_e_3}--\eqref{si_e_4} to deduce that
\begin{equation}\label{si_e_7}
 \int_0^t \sum_{i=1}^n Y_i \cdot S_i \le \frac{1}{2} \int_0^t \int_\Omega \zeta_0 + \int_0^t \left( C(1+ M(s)) \int_\Omega \nu_0(s)\right) ds.
\end{equation}
Notice that the right side of \eqref{si_e_7} is finite since $\nu_0 \in L^\infty([0,T_*];\mathcal{M}(\Omega))$ and $\zeta_0$ has finite mass.  The bound \eqref{si_e_7} proves \eqref{si_e_0} in the case that $Y_i \in C^0$.  If instead $Y_i \in L^2$, then we let $Y_i^\lambda \in C^\infty$ be a smooth mollification of $Y_i$ so that $Y_i^\lambda \to Y_i$ in $L^2$, and we apply \eqref{si_e_7} to $Y_i^\lambda$ to get
\begin{equation}\label{si_e_8}
 \int_0^t \sum_{i=1}^n Y_i^\lambda \cdot S_i \le \frac{1}{2} \int_0^t \int_\Omega \zeta_0 + \int_0^t \left[ C\left(1+ \sum_{i=1}^n \abs{Y_i^\lambda(s)}^2 \right) \int_\Omega \nu_0(s)\right] ds.
\end{equation}
Then since $Y_i^\lambda \to Y_i$ in $L^2$ and $S_i \in L^2$, we may send $\lambda \to 0$ on the left side of \eqref{si_e_8};  since  $\nu_0 \in L^\infty([0,T_*];\mathcal{M}(\Omega))$, we may use dominated convergence to pass to the limit $\lambda \to 0$ on the right of \eqref{si_e_8}.  Taking these limits then yields \eqref{si_e_0}.

\end{proof}

Next we give a result comparing the measure $\zeta$ to $\nu$ through the use of the ``product estimate.''

\begin{prop}\label{dt_lower_bound}
 It holds that
\begin{equation}
 \int_0^{T_*} \sum_{i=1}^n \zeta_i \ge \int_0^{T_*} \alpha \pi^2 \sum_{i=1}^n b^2(a_i) \frac{\abs{\dot{a}_i}^2}{\nu_i}.
\end{equation}
\end{prop}
\begin{proof}

Proposition \ref{prod_est} states that
\begin{equation}
	\abs{\int_0^{T_*} \int_\Omega \frac{b V \cdot \psi Y}{2}}  
\le \liminf_{\ep\rightarrow 0} \left( \int_0^{T_*} \int_\Omega \frac{b \abs{\nab_{B_\ep} v_\ep \cdot Y}^2}{\ale} \right)^{1/2}
	\left(\int_0^{T_*} \int_\Omega  \frac{b \abs{ \psi \dt v_\ep}^2}{\ale} \right)^{1/2}  
\end{equation}
for any  $Y\in C^0(\Omega \times [0,T_*];\Rn{2})$ and $\psi \in C^0(\Omega \times[0,T_*])$.  Note that from \eqref{Tadd} we may compute that 
\begin{equation}
 b \abs{\nab_{B_\ep} v_\ep \cdot Y}^2 = T_\ep:Y \otimes Y +  \tilde{g}_\ep \abs{Y}^2 - \abs{\curl B_\ep}^2 \abs{Y}^2,
\end{equation}
from which we deduce, using the structure of $V$ given in Lemma \ref{vortex_path} and the limits given by Lemma \ref{measure_converge}, that 
\begin{equation}\label{dt_l_b_1}
 \abs{\int_0^{T_*}  \sum_{i=1}^n \pi d_i b(a_i) \dot{a}_i^\bot \cdot  Y(a_i) } 
\le  \left( \int_0^{T_*}\int_\Omega \nu \abs{Y}^2 +  T: Y \otimes Y \right)^{1/2}
	\left(\int_0^{T_*} \int_\Omega  \frac{\psi^2 \zeta}{\alpha} \right)^{1/2}.
\end{equation}

We may  apply this result with $\{Y_j,\psi_j\}_{j\in \mathbb{N}}$ a sequence supported in a $\gamma_*/j$ neighborhood  (with $\gamma_*>0$ given by Lemma \ref{vortex_path}) of each of the paths $a_i$ so that 
\begin{equation}
 Y_j(a_i) \rightarrow \pi d_i \frac{b(a_i)}{\nu_i} \dot{a}_i^\bot 
\text{ and } 
\psi_j(a_i) \to 1 \text{ as } j \rightarrow \infty.
\end{equation}
Then  since Proposition \ref{tensor_structure} says that $T_i=0$, we know that
\begin{equation}
 \int_0^{T_*} \int_\Omega \nu \abs{Y_j}^2 +  T: Y_j \otimes Y_j  \rightarrow \pi^2 \int_0^{T_*} \sum_{i=1}^n  b^2(a_i )\frac{\abs{\dot{a}_i}^2}{\nu_i} \text{ as }j \to \infty.
\end{equation}
Also, as $j \to \infty$, 
\begin{equation}
\int_0^{T_*}  \sum_{i=1}^n \pi d_i b(a_i) \dot{a}_i^\bot \cdot Y_j(a_i) ds \rightarrow \pi^2 \int_0^{T_*} \sum_{i=1}^n b^2(a_i) \frac{\abs{\dot{a}_i}^2}{\nu_i},
\end{equation}
and
\begin{equation}
 \int_0^{T_*} \int_\Omega  \frac{\psi_j^2 \zeta}{\alpha}  \to \frac{1}{\alpha} \int_0^{T_*} \sum_{i=1}^n \zeta_i. 
\end{equation}
The result follows by passing to the limit $j\rightarrow \infty$ with these $Y_j,\psi_j$ in \eqref{dt_l_b_1}.

\end{proof}

With this and the previous lemma, we can identify $\nu$ and   find that the energy does not actually increase by $O(\ale)$.

\begin{prop}\label{nu_equality}
We have 
\begin{equation}\label{n_e_0}
 \tilde{F}_\ep(v_\ep,B_\ep)(t) \le \pi\sum_{i=1}^n b(a_i(t)) \ale + o(\ale) \text{ for all } t\in[0,T_*].
\end{equation}
Moreover, $\zeta_0=0$ and for $t\in[0,T_*]$ we have $\nu_0(t) =0$ and 
\begin{equation}
 \nu_i(t) = \pi b(a_i(t)) \text{ for each } i=1,\dotsc,n.
\end{equation}
\end{prop}
\begin{proof}
 Integrating the energy evolution equation \eqref{m_e_e_02} in time between $0$ and $t\le T_*$ yields
\begin{equation}
  \frac{\tilde{F}_\ep(v_\ep,B_\ep)(t)}{\ale} -  \frac{\tilde{F}_\ep(v_\ep,B_\ep)(0)}{\ale} + \int_0^t \int_\Omega  \frac{ \alpha b \abs{\dt v_\ep}^2 + \sigma \abs{\dt B_\ep}^2}{\ale} = \int_0^t \int_\Omega V(v_\ep,B_\ep) \cdot b Z.
\end{equation}
Passing to the limit $\ep \rightarrow 0$ and using  Proposition \ref{dt_lower_bound}  and the well-preparedness assumption \eqref{well_prepared_def} shows that
\begin{equation}\label{n_e_1}
 \int_\Omega \nu(t) - \sum_{i=1}^n \pi b(a_i(0)) + \int_0^t \alpha \pi^2 \sum_{i=1}^n b^2(a_i) \frac{\abs{\dot{a}_i}^2}{\nu_i}  + \int_0^t \int_\Omega \zeta_0 
\le \int_0^t \sum_{i=1}^n 2\pi d_i  b(a_i) Z(a_i)\cdot \dot{a}_i^\bot.
\end{equation}
Since $a_i \in H^1([0,T_*]; \Omega)$ and $b$ is smooth, we have that $b\circ a_i \in H^1([0,T_*];\Rn{})$ and hence is absolutely continuous on $[0,T_*]$ (cf. Theorems 4.2.2/1 and 4.9.1/1 of \cite{ev_gar});  this implies that
\begin{equation}\label{n_e_2}
 - \sum_{i=1}^n \pi b(a_i(0)) =
- \sum_{i=1}^n   \pi b(a_i(t)) + \int_0^t \sum_{i=1}^n 
\pi \nab b(a_i)\cdot \dot{a}_i .
\end{equation}
By taking the dot product of  \eqref{t_d_04} of Lemma \ref{tensor_diverge} with $\dot{a}_i$,  employing  Proposition \ref{velocities} for $p_i=-\nu_i \dot{a}_i$, and setting $T_i=0$ with the help of Proposition \ref{tensor_structure}, we find that
\begin{equation}\label{n_e_3}
\begin{split}
 2\pi d_i b(a_i) Z(a_i) \cdot \dot{a}_i^\bot &= -2\pi d_i b(a_i) Z^\bot(a_i) \cdot \dot{a}_i \\
&=  \dot{a}_i \cdot \left[ \alpha \nu_i \dot{a}_i + \nu_i \nab \log{b(a_i)} + \beta \pi d_i b(a_i)  \dot{a}_i^\bot  + S_i \right] \\
&= \alpha \nu_i \abs{\dot{a}_i}^2 + \nu_i \dot{a}_i \cdot \nab \log{b(a_i)} + \dot{a}_i \cdot S_i.
\end{split}
\end{equation}
Plugging \eqref{n_e_2} and \eqref{n_e_3} into \eqref{n_e_1} and using the fact that $\nu(t) = \nu_0(t) +  \sum \nu_i(t) \delta_{a_i(t)}$, we find after rearranging that 
\begin{multline}\label{n_e_4}
\int_\Omega \nu_0(t) + \sum_{i=1}^n (\nu_i(t) - \pi b(a_i(t))) + \int_0^t \int_\Omega \zeta_0  
\le \alpha \int_0^t \sum_{i=1}^n \abs{\dot{a}_i}^2\left( \nu_i  - \frac{\pi^2 b^2(a_i )}{\nu_i}\right) \\
 +\int_0^t \sum_{i=1}^n  (\dot{a}_i \cdot  \nab \log b(a_i))(\nu_i  - \pi b(a_i)) 
+ \int_0^t \sum_{i=1}^n \dot{a}_i \cdot S_i.
\end{multline}

Define the function $M(t) = \sum_{i=1}^n \abs{\dot{a}_i(t)}^2$.  We now want to estimate the three temporal integrals on the right of \eqref{n_e_4} in terms of $1+ M(t)$.  For the first term, we use  Lemma \ref{lebesgue_decomp} to see that $\pi b(a_i) \le \nu_i(t)$, which allows us bound
\begin{equation}
\nu_i  - \frac{\pi^2 b^2(a_i)}{\nu_i} =(\nu_i - \pi b(a_i)) \frac{\nu_i + \pi b(a_i)}{\nu_i}  \le 2 ( \nu_i - \pi b(a_i)).
\end{equation}
Hence
\begin{equation}\label{n_e_5}
\alpha \int_0^t \sum_{i=1}^n \abs{\dot{a}_i}^2\left( \nu_i  - \frac{\pi^2 b^2(a_i )}{\nu_i}\right) 
\le 2 \alpha \int_0^t M \sum_{i=1}^n ( \nu_i - \pi b(a_i)).
\end{equation}
For the second term we use Cauchy's inequality for $\abs{\dot{a}_i \cdot \nab \log{b(a_i)}} \le C(\pnorm{\nab \log b}{\infty}^2 + M) \le C(1+M)$, which holds for each $i=1,\dotsc,n$ with $C$ independent of time because $b$ is smooth and bounded below by \eqref{b_lower_bound}.  This yields the bound
\begin{equation}\label{n_e_6}
\int_0^t \sum_{i=1}^n  (\dot{a}_i \cdot  \nab \log b(a_i))(\nu_i  - \pi b(a_i)) \le  \int_0^t C(1+M) \sum_{i=1}^n (\nu_i  - \pi b(a_i)).
\end{equation}
To estimate the third term we will use Proposition \ref{Si_T_estimate} with $Y_i = \dot{a}_i \in L^2([0,T_*];\Rn{2})$ to deduce the bound
\begin{equation}\label{n_e_7}
 \int_0^t \sum_{i=1}^n \dot{a}_i \cdot S_i \le \frac{1}{2} \int_0^t \int_\Omega \zeta_0 + \int_0^t C(1+ M) \int_\Omega \nu_0.
\end{equation}

Now we sum the estimates \eqref{n_e_5}, \eqref{n_e_6}, and \eqref{n_e_7} and replace in \eqref{n_e_4} to find that
\begin{multline}
 \int_\Omega \nu_0(t) + \sum_{i=1}^n (\nu_i(t) - \pi b(a_i(t))) + \hal \int_0^t \int_\Omega \zeta_0  \\
\le   \int_0^t C(1+M) \left[\int_\Omega \nu_0 +   \sum_{i=1}^n (\nu_i  - \pi b(a_i)) \right].
\end{multline}
We may view  this as the differential inequality
\begin{equation}\label{n_e_10}
 \mathcal{F}(t) +   \mathcal{G}(t)   \le \int_0^t \mathcal{Q}(s) \mathcal{F}(s) ds
\end{equation}
with 
\begin{equation}
 \mathcal{F}(t) = \int_\Omega \nu_0(t) + \sum_{i=1}^n (\nu_i(t) - \pi b(a_i(t))) \ge 0,
\end{equation}
\begin{equation}
 \mathcal{G}(t) = \hal \int_0^t \int_\Omega \zeta_0 \ge 0, \text{ and } \mathcal{Q}(t) = C(1 + M(t)) \ge 0.
\end{equation}
Note that since $a_i \in H^1([0,T_*];\Omega)$, we have that $M$ and $\mathcal{Q}$ are in $L^1([0,T_*])$.  We may then use Gronwall's inequality on \eqref{n_e_10} to see that
\begin{equation}
 \int_0^t \mathcal{Q}(s) \mathcal{F}(s) ds \le \left( \int_0^0 \mathcal{Q}(s) \mathcal{F}(s) ds \right) \exp\left( \int_0^t \mathcal{Q}(s) ds  \right) =0,
\end{equation}
and hence that
\begin{equation}\label{n_e_11}
 \mathcal{F}(t) +   \mathcal{G}(t)  \le 0.
\end{equation}
Since $\mathcal{F}, \mathcal{G} \ge 0$, we immediately see that
\begin{equation}\label{n_e_12}
 \nu_0(t) =0 \text{ and }  \sum_{i=1}^n (\nu_i(t) - \pi b(a_i(t)))  = 0 \text{ for all } t\in [0,T_*],
\end{equation}
the latter of which implies that $\nu_i(t) = \pi b(a_i(t))$ for $i=1,\dotsc,n$  since $\nu_i(t) \ge \pi b(a_i(t))$.  The estimate \eqref{n_e_0} follows from \eqref{n_e_12} and the definition of $\nu$.   We may also deduce from \eqref{n_e_11} that that $\zeta_0 =0$.  
\end{proof}

Since $\nu_0(t)dt$ and $\zeta_0$ vanish, we find that other terms vanish as well.

\begin{cor}\label{other_vanish}
 We have that $T_0 =0$, $S_0 = 0$, $p_0 = 0$, and $\sigma \curl{B_*} \dt B_*^\bot =0$.
\end{cor}

\begin{proof}
Since $\nu_0(t)dt =0$ and $\zeta_0 =0$, the vanishing of these terms follows immediately from the estimates of Lemmas \ref{decomp_estimates} and \ref{S0_estimate}.
\end{proof}

 We can use the vanishing of $T_0$, $S_0$ and  $T_i$ for $i=1,\dotsc,n$ to show that $S_i=0$ for $i=1,\dotsc,n$.  This in turn allows us to compute the dynamical law for $a_i$ and complete the proof of Theorem \ref{dynamics-intro}.

\begin{prop}\label{dynamics} 
The vortex trajectories satisfy $a_i \in C^{1}([0,T_*];\Omega)$ for $i=1,\dotsc,n$.  Moreover, they obey the dynamical law
\begin{equation}\label{dynamics_0}
\begin{split}
\a\dot{a}_i + d_i \beta \dot{a}_i^\bot  &= -2 d_i (\nab^\bot \fb(a_i) - \xb^\bot(a_i)) - \nab \log b(a_i) \\
& = -2 d_i (\nab^\bot \fb(a_i) + \nab \xib (a_i)) - \nab \log b(a_i) \\
& =  -2 d_i \left(\frac{ \sigma \nab^\bot \phib(a_i) + \nab \hb(a_i)  }{b(a_i)} \right)- \nab \log b(a_i).
\end{split}
\end{equation} 
\end{prop}

\begin{proof}

According to Lemma \ref{tensor_diverge} and Corollary \ref{other_vanish}, we have that
\begin{equation}
 \diverge{T} = S_0 + \sum_{i=1}^n S_i(t) \delta_{a_i(t)}dt = \sum_{i=1}^n S_i(t) \delta_{a_i(t)}dt.
\end{equation}
On the other hand, Lemma \ref{lebesgue_decomp}, Proposition \ref{tensor_structure}, and Corollary \ref{other_vanish} imply that
\begin{equation}
 \diverge{T} = \diverge\left( T_0 + \sum_{i=1}^n  T_i(t) \delta_{a_i(t)}dt \right) = 0.
\end{equation}
Equating these two, we find that
\begin{equation}
  \sum_{i=1}^n S_i(t) \delta_{a_i(t)} dt=0,
\end{equation}
and hence  $S_i=0$ for a.e. $t\in[0,T_*]$.  We use $S_i=0$ in the equation for $S_i$ given in \eqref{t_d_04} of Lemma \ref{tensor_diverge}, and then we substitute in the values of $p_i$ and $\nu_i$ given in Propositions \ref{velocities} and  \ref{nu_equality} to deduce that for a.e. $t \in [0,T_*]$, 
\begin{equation}\label{dynamics_1}
 \a\dot{a}_i + d_i \beta \dot{a}_i^\bot  = -2 d_i Z^\bot(a_i)  - \nab \log b(a_i).
\end{equation}

We may solve \eqref{dynamics_1} for $\dot{a}_i$ to see that \eqref{dot_solved} holds.  According to Remark \ref{Z_smooth}, the vector field $Z$ is smooth and bounded, while \eqref{b_lower_bound} implies that $\nab \log(b)$ is smooth and bounded; then  since $a_i \in C^{0,1/2}$, the right side of \eqref{dot_solved} is a continuous function of $t$ in $[0,T_*]$.  Plugging \eqref{dot_solved} into the equation
\begin{equation}
 a_i(t) = a_i(0) + \int_0^t \dot{a}_i(s) ds,
\end{equation}
which follows from the absolute continuity of $a_i$ (again, cf. Theorem 4.9.1/1 of \cite{ev_gar}), we see that $a_i(t)-a_i(0)$ is the integral of a continuous function and is thus classically differentiable.  Moreover, \eqref{dot_solved} implies  that $\pnorm{\dot{a}_i}{\infty} \le C$ so that $a_i \in C^{1}([0,T_*];\Omega)$.  

The first equation in \eqref{dynamics_0} follows from \eqref{dynamics_1} by using $Z= \nab \fb - \xb$, which is the definition of $Z$ given in Lemma \ref{specific_reformulation}.  The second equality in \eqref{dynamics_0} follows from the definition of $\xb$ \eqref{X_def} and equation \eqref{propeX}.  The third follows from  Lemma \ref{X_f_relation}.

\end{proof}

Although we have only worked on the interval $[0,T_*]$, the dynamics can be extended  until a collision occurs. Indeed, Proposition \ref{nu_equality} says that
\begin{equation}
 \tilde{F}_\ep(v_\ep, B_\ep) (T_*) \le \pi\sum_{i=1}^n b(a_i(T_*)) \ale + o(\ale).
\end{equation}
We can then run all of the above analysis again, starting from time $t=T_*$.  The only obstacle to running this iteration forever is the possibility of a vortex collision or a vortex exiting the domain. Hence the maximal time of validity of the theorem is the first occurrence of such a collision or an exit under the law \eqref{dynalaw}. Theorem \ref{dynamics-intro} is proved.

The proof of Theorem \ref{th2} can be obtained following the same steps (things are actually made simpler by the absence of the gauge $B$) according to the sketch given in Section \ref{glsimple}.
Details are left to the reader.

\section*{Acknowledgments}

We would like to thank an anonymous referee for their careful reading and for pointing out a flaw in one of the arguments used in the first draft of this paper.

\pagebreak

\vskip 1cm
\noindent
{\sc Sylvia Serfaty}\\
UPMC Univ. Paris 06, UMR 7598 Laboratoire Jacques-Louis Lions,\\
 Paris, F-75005 France ;\\
 CNRS, UMR 7598 LJLL, Paris, F-75005 France \\ 
 \&  Courant Institute, New York University\\
251 Mercer St., New York, NY  10012, USA\\
{\tt serfaty@ann.jussieu.fr}
\vskip 1cm
\noindent
{\sc Ian Tice}\\
Brown University, Division of Applied Mathematics\\
 182 George St., Providence, RI 02912, USA \\ 
{\tt tice@dam.brown.edu}

\end{document}